\documentclass[a4paper,12pt]{amsart}
\usepackage{amsmath}
\usepackage{amsmath}
\usepackage{amssymb}
\usepackage{amscd}
\usepackage{mathrsfs}
\usepackage[all]{xy}
\usepackage[dvipdfm]{graphicx}
\def\mathcal{\mathscr}
\everymath{\displaystyle}
\newtheorem{thm}{Theorem}[section]
\newtheorem{lem}[thm]{Lemma}
\newtheorem{cor}[thm]{Corollary}
\newtheorem{prop}[thm]{Proposition}

\theoremstyle{definition}

\newtheorem{rem}[thm]{Remark}

\newtheorem{defn}[thm]{Definition}

\newcommand{\mca}[1]{{\mathcal{#1}}}

\def\Z{{\mathbb Z}}
\def\C{{\mathbb C}}
\def\R{{\mathbb R}}
\def\con{\textit{\rm con}} 
\def\CF{\text{\rm CF}}
\def\const{\text{\rm const}}
\def\CM{\text{\rm CM}}
\def\CZ{\text{\rm CZ}}
\def\dist{\text{\rm dist}}
\def\ev{\text{\rm ev}}
\def\FH{\text{\rm FH}}
\def\HF{\text{\rm HF}}
\def\HM{\text{\rm HM}}
\def\id{\text{\rm id}}
\def\Im{\text{\rm Im}\,}
\def\ind{\text{\rm ind}}
\def\interior{\text{\rm int}}

\def\Morse{\text{\rm Morse}\,}

\def\SH{\text{\rm SH}}
\def\Star{\text{\rm Star}} 

\def\supp{\text{\rm supp}}
\def\std{\text{\rm std}}
\def\vol{\text{\rm vol}}
\begin{document}
\pagestyle{plain}
\thispagestyle{plain}

\title[Symplectic homology of disc cotangent bundles of domains in Euclidean space]
{Symplectic homology of disc cotangent bundles of domains in Euclidean space}

\author[Kei Irie]{Kei Irie}
\address{Research Institute for Mathematical Sciences, Kyoto University,
Kyoto 606-8502, Japan}
\email{iriek@kurims.kyoto-u.ac.jp}

\date{\today}

\begin{abstract}
Let $V$ be a bounded domain with smooth boundary in $\R^n$, and $D^*V$ denote its disc cotangent bundle. 
We compute symplectic homology of $D^*V$, in terms of relative homology of loop spaces on the closure of $V$. 
We use this result to show that the Floer-Hofer capacity of $D^*V$ is between $2r(V)$ and $2(n+1)r(V)$, where $r(V)$ 
denotes the inradius of $V$. 
As an application, we study periodic billiard trajectories on $V$. 
\end{abstract}

\maketitle

\section{Introduction}

\subsection{Main result} 
Let us consider the symplectic vector space $T^*\R^n$, with coordinates $p_1,\ldots,p_n,q_1,\ldots,q_n$ and 
the standard symplectic form $\omega_n:= dp_1 \wedge dq_1 + \cdots + dp_n \wedge dq_n$. 
For any bounded open set $U \subset T^*\R^n$ and real numbers $a<b$, one can define a $\Z_2$-module
$\SH_*^{[a,b)}(U)$, which is called \textit{symplectic homology}. 
This invariant was introduced in \cite{FH}. 
Our first goal is to compute $\SH_*^{[a,b)}(U)$, when $U$ is a disk cotangent bundle of a domain in $\R^n$. 

First let us fix notations. 
For any domain (i.e. connected open set) $V \subset \R^n$, its disc cotangent bundle $D^*V \subset T^*\R^n$ is defined as
\[
D^*V:= \{ (q,p) \in T^*\R^n \mid q \in V, |p| < 1 \}. 
\]
We use the following notations for loop spaces: 
\begin{itemize}
\item $\Lambda(\R^n):= W^{1,2}(S^1,\R^n)$, where $S^1:=\R/\Z$. 
\item $\Lambda^{<a}(\R^n):= \{ \gamma \in \Lambda(\R^n) \mid \text{length of $\gamma$} < a \}$. 
\item For any subset $S \subset \R^n$, we set 
\[
\Lambda(S):= \{ \gamma \in \Lambda(\R^n) \mid \gamma(S^1) \subset S\}, \quad 
\Lambda^{<a}(S):= \Lambda(S) \cap \Lambda^{<a}(\R^n). 
\]
\end{itemize}

Then, the main result in this note is the following: 

\begin{thm}\label{thm:main}
Let $V$ be a bounded domain with smooth boundary in $\R^n$, and $\bar{V}$ denote its closure in $\R^n$. 
For any $a<0$ and $b>0$, there exists a natural isomorphism 
\[
\SH^{[a,b)}_*(D^*V) \cong H_*\bigl(\Lambda^{<b}(\bar{V}), \Lambda^{<b}(\bar{V}) \setminus \Lambda(V) \bigr).
\]
Moreover, for any $0<b^-<b^+$ the following diagram commutes:
\[
\xymatrix{
\SH_*^{[a,b^-)}(D^*V) \ar[r]^-{\cong} \ar[d] &  H_*\bigl(\Lambda^{<b^-}(\bar{V}), \Lambda^{<b^-}(\bar{V}) \setminus \Lambda(V) \bigr)\ar[d] \\
\SH_*^{[a,b^+)}(D^*V) \ar[r]_-{\cong}       & H_*\bigl(\Lambda^{<b^+}(\bar{V}), \Lambda^{<b^+}(\bar{V}) \setminus \Lambda(V) \bigr).
}
\]
The left vertical arrow is a natural map in symplectic homology, and the right vertical arrow is induced by inclusion. 
\end{thm}

\subsection{Floer-Hofer capacity and periodic billiard trajectories}

By using symplectic homology, one can define the \textit{Floer-Hofer capacity}, which is denoted as $c_\FH$. 
The Floer-Hofer capacity was introduced in \cite{FHW}. We recall its definition in Section 2.4.
The Floer-Hofer capacity of a disk cotangent bundle $D^*V$ is important in the study of \textit{periodic billiard trajectories} on $V$
(for precise definition, see Definition \ref{def:billiard}):

\begin{prop}\label{prop:billiard}
Let $V$ be a bounded domain with smooth boundary in $\R^n$. 
Then, there exists a periodic billiard trajectory $\gamma$ on $V$ with at most $n+1$ bounce times such that 
$\text{length of $\gamma$} = c_\FH(D^*V)$. 
\end{prop}

\begin{rem}
The idea of using symplectic capacities to study periodic billiard trajectory is due to Viterbo \cite{Vit2}.
See also \cite{AvOs}, in which a result similar to Proposition \ref{prop:billiard} (Theorem 2.13 in \cite{AvOs}) is proved. 
Proposition \ref{prop:billiard} is essentially the same as Theorem 13 in \cite{Ir}. 
However, our formulation of symplectic homology in this note is a bit different from that in \cite{Ir}, in which we used Viterbo's symplectic homology \cite{Vit1}.  
Hence we include a proof of Proposition \ref{prop:billiard} in Section 6, for the sake of completeness. 
\end{rem} 

Given Proposition \ref{prop:billiard}, it is natural to ask if one can compute $c_\FH(D^*V)$ by using only 
elementary (i.e. singular) homology theory. 
The following corollary of our main result gives an answer to this question. 
For any $x \in \bar{V}$, $c_x$ denotes the constant loop at $x$. 

\begin{cor}\label{cor:FH}
Let $V$ be a bounded domain with smooth boundary in $\R^n$, and $b>0$. 
Let us define 
$\iota^b: (\bar{V}, \partial V) \to \bigl( \Lambda^{<b}(\bar{V}), \Lambda^{<b}(\bar{V}) \setminus \Lambda(V) \bigr)$ by 
$\iota^b(x):= c_x$. Denote by $(\iota^b)_*$ the map on homology induced by $\iota^b$. Then, 
\begin{align*}
&c_\FH(D^*V)=  \\
&\quad \inf \{ b \mid \text{$(\iota^b)_n: H_n(\bar{V}, \partial V) \to H_n(\Lambda^{<b}(\bar{V}), \Lambda^{<b}(\bar{V}) \setminus \Lambda(V))$ vanishes} \}. 
\end{align*}
\end{cor}

To prove Corollary \ref{cor:FH}, we need to combine our main result Theorem \ref{thm:main} with results in \cite{Her}. 
Corollary \ref{cor:FH} is proved in Section 6. 

\subsection{Floer-Hofer capacity and inradius} 

Using Corollary \ref{cor:FH}, one can obtain a quite good estimate of $c_\FH(D^*V)$ by using the inradius of $V$. 
First let us define the notion of the inradius: 

\begin{defn}\label{defn:inrad}
Let $V$ be a domain in $\R^n$. The \textit{inradius} of $V$, which is denoted as $r(V)$, is the supremum of radii of balls in $V$. 
In other words, $r(V):= \sup_{x \in V} \dist(x, \partial V)$. 
\end{defn}

Our estimate of the Floer-Hofer capacity is the following: 

\begin{thm}\label{thm:inrad}
Let $V$ be a bounded domain with smooth boundary in $\R^n$. 
Then, there holds 
$2 r(V) \le c_\FH(D^*V) \le 2(n+1) r(V)$.  
\end{thm}

Combined with Proposition \ref{prop:billiard}, Theorem \ref{thm:inrad} implies the following result: 

\begin{cor}\label{cor:billiard}
Let $V$ be a bounded domain with smooth boundary in $\R^n$. 
There exists a periodic billiard trajectory on $V$ with at most $n+1$ bounce times and 
length between $2r(V)$ and $2(n+1)r(V)$. 
\end{cor} 
\begin{rem} 
Let $\xi(V)$ denote the infimum of the lengths of periodic billiard trajectories on $V$. 
Corollary \ref{cor:billiard} shows that $\xi(V) \le 2(n+1)r(V)$. 
When $V$ is \textit{convex}, this result was already established as Theorem 1.3 in \cite{AvOs}. 
On the other hand, the main result in \cite{Ir} is that $\xi(V) \le \const_n r(V)$ for any domain $V$ with smooth boundary in $\R^n$. 
A weaker result $\xi(V) \le \const_n \vol(V)^{1/n}$ was obtained in \cite{Vit2}, \cite{FGS}. 
\end{rem} 

Theorem \ref{thm:inrad} is proved in Section 7. 
Here we give a short comment on the proof. 
Actually, the lower bound is immediate from Corollary \ref{cor:FH}, and the issue is to prove the upper bound. 
By Corollary \ref{cor:FH}, it is enough to show that if $b>2(n+1)r(V)$, then 
$(\iota^b)_*[(\bar{V},\partial V)]=0$. 
We will prove this by constructing a $(n+1)$-chain in $\bigl(\Lambda^{<b}(\bar{V}), \Lambda^{<b}(\bar{V}) \setminus \Lambda(V)\bigr)$ 
which bounds $(\bar{V},\partial V)$. Details are carried out in Section 7. 

\subsection{Organization of the paper}

In Section 2, we recall the definition and main properties of symplectic homology, following \cite{FH}. 
In Section 3, we recall Morse theory for Lagrangian action functionals on loop spaces, following \cite{AbM}, \cite{AbS2}. 
The goal in these sections is to fix a setup for the arguments in Sections 4, 5, 6. 

In Section 4, we prove our main result Theorem \ref{thm:main}. 
The proof consists of two steps: 

\textbf{Step1:} In Theorem \ref{thm:key}, we prove an isomorphism between Floer homology 
of a quadratic Hamiltonian on $T^*\R^n$ and Morse homology of its fiberwise Legendre transform. 

\textbf{Step2:} By taking a limit of Hamiltonians, we deduce Theorem \ref{thm:main} from Theorem \ref{thm:key}. 

Our proof of Theorem \ref{thm:key} is based on \cite{AbS1}:  
we construct an isomorphism by using so called \textit{hybrid moduli spaces}. 
However, since we will work on $T^*\R^n$, proofs of various $C^0$- estimates for (hybrid) Floer trajectories are not automatic. 
Techniques in \cite{AbS1} (in which the authors are working on cotangent bundles of \textit{compact} manifolds) do not seem to work directly in our setting. 
To prove $C^0$- estimates for Floer trajectories in our setting, 
we combine techniques in \cite{AbS1} and \cite{FH}.  
Proofs of $C^0$- estimates are carried out in Section 5. 

In Section 6, we discuss the Floer-Hofer capacity and periodic billiard trajectories. The goal of this section is to prove 
Proposition \ref{prop:billiard} and Corollary \ref{cor:FH}. 

In Section 7, we prove Theorem \ref{thm:inrad}. 
This section can be read almost independently from the other parts of the paper.  

\section{Symplectic homology}
We recall the definition and main properties of symplectic homology. 
We basically follow \cite{FH}. 

\subsection{Hamiltonian}
For $H \in C^\infty(T^*\R^n)$, its \textit{Hamiltonian vector field} $X_H$ is defined as 
$\omega_n(X_H, \cdot \, ) = -dH(\, \cdot \,)$.

For $H \in C^\infty(S^1 \times T^*\R^n)$ and $t \in S^1$, $H_t \in C^\infty(T^*\R^n)$ is defined as $H_t(q,p):=H(t,q,p)$. 
$\mca{P}(H)$ denotes the set of periodic orbits of $(X_{H_t})_{t \in S^1}$, i.e.
\[
\mca{P}(H):= \{ x \in C^\infty(S^1,T^*\R^n) \mid \dot{x}(t) = X_{H_t}(x(t)) \}.
\]
$x \in \mca{P}(H)$ is called \textit{nondegenerate} if $1$ is not an eigenvalue of 
the Poincar\'{e} map associated with $x$. 
We introduce the following conditions on $H \in C^\infty(S^1 \times T^*\R^n)$: 
\begin{enumerate}
\item[(H0):] Every element in $\mca{P}(H)$ is nondegenerate. 
\item[(H1):] There exists $a \in (0,\infty) \setminus \pi \Z$ such that 
$\sup_{t \in S^1} \| H_t - Q^a\|_{C^1(T^*\R^n)} < \infty$, where $Q^a(q,p):=a(|q|^2+|p|^2)$. 
\end{enumerate}

\begin{rem}\label{rem:HaminFH}
The class of Hamiltonians considered in this note is a bit different from that in \cite{FH}. 
To put it more precisely, (H1) is more restrictive than conditions (6), (7) in \cite{FH}. 
On the other hand, we do not need condition (8) in \cite{FH}. 
It is easy to see that our definition of symplectic homology is equivalent to that in \cite{FH}, 
see Remark \ref{rem:SHinFH}. 
\end{rem}

\begin{lem}\label{lem:orbit}
For any $H \in C^\infty(S^1 \times T\R^n)$ which satisfies (H1), $\mca{P}(H)$ is $C^0$-bounded. 
In particular, if $H$ also satisfies (H0), then $\mca{P}(H)$ is a finite set. 
\end{lem}
\begin{proof}
Suppose that there exists $H \in C^\infty(S^1 \times T^*\R^n)$ which satisfies (H1) and $\mca{P}(H)$ is \textit{not} $C^0$-bounded. 
Then there exists a sequence $(x_j)_{j=1,2,\ldots}$ in $\mca{P}(H)$ such that $R_j:= \max_{t \in S^1} |x_j(t)|$ goes to $\infty$ as $j \to \infty$. 
Define $v_j:S^1 \to T^*\R^n$ and $h^j \in C^\infty(S^1 \times T^*\R^n)$ by 
\[
v_j(t):= x_j(t)/R_j, \qquad
h^j(t,q,p):= H(t, R_jq, R_jp)/{R_j^2}. 
\] 
It is easy to show that $v_j \in \mca{P}(h^j)$. 
Moreover, since $\sup_{t \in S^1} \| dH_t - dQ^a \|_{C^0} < \infty$, 
\begin{equation}\label{eq:hjQ}
\lim_{j \to \infty} \sup_{t \in S^1} \| dh^j_t -dQ^a\|_{C^0} = 0. 
\end{equation}
By definition, $\max_{t \in S^1} |v_j(t)|=1$. In particular, $(v_j)_j$ is $C^0$-bounded. 
Moreover, since $\partial_t v_j = X_{h^j_t}(v_j)$, (\ref{eq:hjQ}) shows that 
$(v_j)_j$ is $C^1$-bounded. 
Hence, up to a subsequence, 
$(v_j)_j$ converges in $C^0(S^1, T^*\R^n)$. We denote the limit by $v$.

By the triangle inequality, 
\begin{align*}
&\int_0^1 |X_{h^j_t}(v_j(t)) - X_{Q^a}(v(t))| \, dt 
\le  \\
& \qquad \int_0^1 |X_{h^j_t}(v_j(t)) - X_{Q^a}(v_j(t))| \, dt + 
\int_0^1 |X_{Q^a}(v_j(t)) - X_{Q^a}(v(t))| \, dt.
\end{align*}
As $j \to \infty$, the first term on the RHS goes to $0$ by (\ref{eq:hjQ}), and the second term on the RHS goes to $0$ since $v_j$ converges to $v$ in $C^0$. 
Therefore, for any $0 \le t_0 \le 1$, 
\begin{align*}
v(t_0) - v(0)&= \lim_{j \to \infty} v_j(t_0)  -v_j(0) \\
&= \lim_{j \to \infty} \int_0^{t_0} X_{h^j_t} (v_j(t))\, dt 
= \int_0^{t_0}  X_{Q^a}(v(t))\, dt, 
\end{align*}
hence $v \in \mca{P}(Q^a)$. 
On the other hand, it is clear that 
$\max_{t \in S^1} |v(t)|=1$. 
This is a contradiction, since $a \notin \pi \Z$ implies that the only element in $\mca{P}(Q_a)$ is the constant loop at $(0,\ldots,0)$. 
\end{proof} 

$H \in C^\infty(S^1 \times T^*\R^n)$ is called \textit{admissible} if it satisfies (H0) and (H1).

\subsection{Truncated Floer homology}

Let $J=(J_t)_{t \in S^1}$ be a time dependent almost complex structure on $T^*\R^n$, such that:
\begin{enumerate}
\item[(J1):] For any $t \in S^1$, $J_t$ is compatible with $\omega_n$, i.e. 
$g_{J_t}(\xi,\eta):= \omega_n(\xi, J_t \eta)$ is a Riemannian metric on $T^*\R^n$. 
\end{enumerate}
Let $H \in C^\infty(S^1 \times T^*\R^n)$ be an admissible Hamiltonian. 
For any $x_-, x_+  \in \mca{P}(H)$, we introduce the Floer trajectory space in the usual manner: 
\begin{align*}
&\mca{M}_{H,J}(x_-, x_+):= \\
&\qquad \{u: \R \times S^1 \to T^*\R^n \mid 
\partial_s u - J_t(\partial_t u - X_{H_t}(u)) =0,\, \lim_{s \to \pm \infty} u(s) = x_{\pm} \}. 
\end{align*}
We set $\bar{\mca{M}}_{H,J}(x_-, x_+):= \mca{M}_{H,J}(x_-,x_+)/\R$, where $\R$ acts on $\mca{M}_{H,J}(x_-,x_+)$ by shift in the 
$s$-variable.

The \textit{standard complex structure} $J_\std$ on $T^*\R^n$ is defined as 
\[
J_{\std}(\partial_{p_i}):= \partial_{q_i}, \qquad
J_{\std}(\partial_{q_i}):= -\partial_{p_i}.
\]
Now we state our first $C^0$-estimate. 
It is proved in Section 5. 

\begin{lem}\label{lem:C0_1}
There exists a constant $\varepsilon>0$ which satisfies the following property: 
\begin{quote}
For any admissible Hamiltonian $H \in C^\infty(S^1 \times T^*\R^n)$ and $J = (J_t)_{t \in S^1}$ which satisfies 
(J1) and $\sup_t \| J_t - J_{\std} \|_{C^0} < \varepsilon$, 
$\mca{M}_{H,J}(x_-,x_+)$ is $C^0$-bounded for any $x_-, x_+ \in \mca{P}(H)$. 
\end{quote}
\end{lem}

We recall the definition of Floer homology. 
For any $\gamma \in C^\infty(S^1,T^*\R^n)$, we set 
\[
\mca{A}_H(\gamma):= \int_{S^1} \gamma^*\biggl(\sum_i p_i dq_i \biggr) - H(t,\gamma(t)) \, dt. 
\]
For real numbers $a<b$, the 
\textit{Floer chain complex} 
$\CF^{[a,b)}_*(H)$ is the free $\Z_2$ module generated by $\{ \gamma \in \mca{P}(H) \mid \mca{A}_H(\gamma)  \in [a,b) \}$, 
indexed by the Conley-Zehnder index $\ind_{\CZ}$. 
For the definition of the Conley-Zehnder index, see Section 1.3 in \cite{FH}. 

Suppose that $J=(J_t)_{t \in S^1}$ satisfies (J1) and each $J_t$ is sufficiently close to $J_\std$. 
Lemma \ref{lem:C0_1} shows that for generic $J$, 
$\bar{\mca{M}}_{H,J}(x_-,x_+)$ is a compact $0$-dimensional 
manifold for any $x_-, x_+  \in \mca{P}(H)$ such that $\ind_\CZ(x_-) - \ind_\CZ(x_+)=1$. 
We can thus define the Floer differential $\partial_{H,J}$ on $\CF^{[a,b)}_*(H)$ as 
\[
\partial_{H,J}\bigl([x_-] \bigr):= \sum_{\ind_{\CZ}(x_+)=\ind_{\CZ}(x_-)-1} \sharp \bar{\mca{M}}_{H,J}(x_-, x_+) \cdot [x_+].
\]
The usual gluing argument shows that $\partial_{H,J}^2 =0$. 
$\HF^{[a,b)}_*(H,J):= H_*(\CF^{[a,b)}_*(H), \partial_{H,J})$ is called \textit{truncated Floer homology}. 

\subsection{Symplectic homology} 
Suppose that we are given the following data: 
\begin{itemize}
\item Admissible Hamiltonians $H^-, H^+ \in C^\infty(S^1 \times T^*\R^n)$
\item $J^-=(J^-_t)_{t \in S^1}$, $J^+=(J^+_t)_{t \in S^1}$, which satisfy (J1). Moreover, all $J^-_t$, $J^+_t$ are sufficiently close to $J_\std$. 
\end{itemize}
We assume that $\HF^{[a,b)}_*(H^-,J^-)$, $\HF^{[a,b)}_*(H^+, J^+)$ are well-defined. 
If $H^- \le H^+$, i.e. $H^-(t,q,p) \le H^+(t,q,p)$ for any $t \in S^1$ and $(q,p) \in T^*\R^n$, one can define \textit{monotonicity homomorphism}
\[
\HF_*^{[a,b)}(H^-,J^-) \to \HF_*^{[a,b)}(H^+,J^+)
\]
in the following way. 

First we introduce the following conditions on $H \in C^\infty(\R \times S^1 \times T^*\R^n)$: 

\begin{enumerate}
\item[(HH1):] There exists $s_0>0$ such that $H(s,t,q,p)=\begin{cases} H(s_0, t,q,p) &(s \ge s_0) \\ H(-s_0, t,q,p) &( s \le -s_0) \end{cases}$.
\item[(HH2):] $\partial_s H(s,t,q,p) \ge 0$ for any $(s,t,q,p) \in \R \times S^1 \times T^*\R^n$. 
\item[(HH3):] There exists $a(s) \in C^\infty(\R)$ such that:
\begin{itemize}
\item $a'(s) \ge 0$ for any $s$. 
\item $a(s) \in \pi \Z \implies a'(s) >0$. 
\item Setting $\Delta(s,t,q,p):=H(s,t,q,p)- Q^{a(s)}(q,p)$, there holds 
\[
\sup_{(s,t)} \| \Delta_{s,t} \|_{C^1(T^*\R^n)} < \infty, \qquad
\sup_{(s,t)} \| \partial_s \Delta_{s,t} \|_{C^0(T^*\R^n)} < \infty. 
\]
\end{itemize}
\end{enumerate}

If $H$ satisfies (HH1), (HH2), (HH3) and $H^{\pm}_t = H_{\pm s_0, t}$, $H$ is called a homotopy from $H^-$ to $H^+$. 
For any $H^-$ and $H^+$ such that $H^- \le H^+$, there exists a homotopy from $H^-$ to $H^+$. 
In fact, take $\rho \in C^\infty(\R)$ such that 
\begin{align*}
&s \ge 1 \implies \rho(s)=1, \quad
s \le 0 \implies \rho(s)=0, \\
&0<s<1 \implies \rho(s) \in (0,1), \, \rho'(s)>0.
\end{align*}
Then $H(s,t,q,p):= \rho(s) H^+(t,q,p) + (1-\rho(s)) H^-(t,q,p)$ is a homotopy from $H^-$ to $H^+$. 

Next we introduce conditions on 
$J=(J_{s,t})_{(s,t) \in \R \times S^1}$, a family of almost complex structures on $T^*\R^n$ parametrized by $\R \times S^1$: 
\begin{enumerate}
\item[(JJ1):] There exists $s_1>0$ such that 
$J_{s,t}(q,p) = \begin{cases} J_{s_1,t}(q,p) &(s \ge s_1) \\ J_{-s_1,t}(q,p) &(s \le -s_1) \end{cases}$.
\item[(JJ2):] For any $(s,t) \in \R \times S^1$, $J_{s,t}$ is compatible with $\omega_n$. 
\end{enumerate}
If $J$ satisfies (JJ1), (JJ2) and $J^\pm_t= J_{\pm s_1, t}$, $J$ is called a homotopy from $J^-$ to $J^+$. 

Let $H \in C^\infty(\R \times S^1 \times T^*\R^n)$ be a homotopy from $H^-$ to $H^+$, and 
$J=(J_{s,t})_{(s,t) \in \R \times S^1}$ be a homotopy from $J^-$ to $J^+$. 
For any $x_- \in \mca{P}(H^-)$ and $x_+ \in \mca{P}(H^+)$, we define
\begin{align*}
&\mca{M}_{H,J}(x_-,x_+):= \\
&\quad \{ u: \R \times S^1 \to T^*\R^n \mid \partial_s u - J_{s,t}(\partial_t u - X_{H_{s,t}}(u))=0, \,  \lim_{s \to \pm \infty} u(s)=x_{\pm}\}.
\end{align*}
Now we state our second $C^0$-estimate. It is proved in Section 5. 

\begin{lem}\label{lem:C0_2}
There exists a constant $\varepsilon>0$ which satisfies the following property: 
\begin{quote}
If $J=(J_{s,t})_{(s,t) \in \R \times S^1}$ satisfies $\sup_{s,t} \| J_{s,t} - J_{\std} \|_{C^0} < \varepsilon$, 
$\mca{M}_{H,J}(x_-,x_+)$ is $C^0$-bounded for any $x_- \in \mca{P}(H^-)$, $x_+ \in \mca{P}(H^+)$. 
\end{quote}
\end{lem}

Lemma \ref{lem:C0_2} shows that, if $J$ is generic and all $J_{s,t}$ are sufficiently close to $J_\std$, 
$\mca{M}_{H,J}(x_-,x_+)$ is a compact $0$-dimensional manifold for any $x_- \in \mca{P}(H^-)$, $x_+ \in \mca{P}(H^+)$ such that
$\ind_{\CZ}(x^-) = \ind_{\CZ}(x^+)$. 
We define $\Phi: \CF^{[a,b)}_*(H^-,J^-) \to \CF^{[a,b)}_*(H^+,J^+)$ by 
\[
\Phi\bigl([x^-]\bigr):= \sum_{\ind_{\CZ}(x^+) = \ind_\CZ(x^-)} \sharp \mca{M}_{H,J}(x^-,x^+) \cdot [x^+].
\]
The usual gluing argument shows that $\Phi$ is a chain map. 
The monotonicity homomorphism 
\[
\Phi_*:  \HF^{[a,b)}_*(H^-,J^-) \to \HF^{[a,b)}_*(H^+,J^+) 
\]
is the homomorphism on homology induced by $\Phi$. 
One can show that $\Phi_*$ does not depend on the choices of $H$ and $J$, 
see Section 4.3 in \cite{FH}. 

\begin{rem}
Let $H$ be an admissible Hamiltonian, and 
$J^0, J^1$ be $S^1$- dependent almost complex structures such that 
$\HF^{[a,b)}_*(H,J^0)$, $\HF^{[a,b)}_*(H,J^1)$ are well-defined. 
Then, one can show that the monotonicity homomorphism 
$\HF^{[a,b)}_*(H,J^0) \to \HF^{[a,b)}_*(H,J^1)$ is an isomorphism. 
Hence $\HF_*^{[a,b)}(H,J)$ does not depend on $J$, and we denote it by $\HF_*^{[a,b)}(H)$. 
Moreover, for two admissible Hamiltonians $H^-, H^+$ satisfying 
$H^- \le H^+$, the monotonicity homomorphism 
$\HF^{[a,b)}_*(H^-) \to \HF^{[a,b)}_*(H^+)$ is well-defined. 
\end{rem}

We define symplectic homology. 
Let $U$ be a bounded open set in $T^*\R^n$. 
Let $\mca{H}_U$ denote the set consisting of admissible Hamiltonians $H$ such that 
$H|_{S^1 \times \bar{U}} < 0$. 
$\mca{H}_U$ is a directed set with relation
\[
H^- \le H^+ \iff H^-(t,q,p) \le H^+(t,q,p) \qquad \bigl( \forall (t,q,p) \in S^1 \times T^*\R^n). 
\]
Then, for any $-\infty < a < b < \infty$, we define symplectic homology $\SH_*^{[a,b)}(U)$ by   
\[
\SH^{[a,b)}_*(U):= \varinjlim_{H \in \mca{H}_U} \HF^{[a,b)}_*(H). 
\]

If $U \subset V$, then obviously $\mca{H}_V \subset \mca{H}_U$. Hence there exists a 
natural homomorphism 
\[
\SH^{[a,b)}_*(V) \to \SH^{[a,b)}_*(U). 
\]
Moreover, for any $a^\pm, b^\pm \in \R$ such that $a^- \le a^+$, $b^- \le b^+$, $a^- < b^-$, $a^+ < b^+$, there exists a natural homomorphism 
$\SH^{[a^-,b^-)}_*(U) \to \SH^{[a^+,b^+)}_*(U)$. 

\begin{rem}\label{rem:SHinFH}
As noted in Remark \ref{rem:HaminFH}, the class of Hamiltonians considered here is different from that in \cite{FH}. 
However, our definition of symplectic homology given above is equivalent to the definition in \cite{FH} (see Section 1.6 in \cite{FH}). 
A key fact is that compact perturbations of quadratic Hamiltonians are admissible both in our sense and sense in \cite{FH}. 
\end{rem}

\subsection{Floer-Hofer capacity} 

Finally, we define the Floer-Hofer capacity, which is originally due to \cite{FHW}. 
For any bounded open set $U$ and $b>0$, we define 
\[
\SH_*^{(0,b)}(U) := \varprojlim_{ \varepsilon \to +0} \SH_*^{[\varepsilon,b)}(U). 
\]
When $U \subset V$, there exists a natural homomorphism $\SH_*^{(0,b)}(V) \to \SH_*^{(0,b)}(U)$. 
For any $p \in T^*\R^n$, we define 
\[
\Theta^b(p):= \varinjlim_{\varepsilon \to +0 } \SH_{n+1}^{(0,b)}(B^{2n}(p:\varepsilon))
\]
where $B^{2n}(p:\varepsilon)$ denotes the open ball in $T^*\R^n$ with center $p$ and radius $\varepsilon$. 
It is known that $\Theta^b(p) \cong \Z_2$, see pp. 603--604 in \cite{FHW}. 

Let $U$ be a bounded domain (hence \textit{connected}) in $T^*\R^n$. 
Taking $p \in U$ arbitrarily, we define the Floer-Hofer capacity of $U$ as 
\[
c_{\FH}(U):= \inf \{ b \mid \text{$\SH_{n+1}^{(0,b)}(U) \to \Theta^b(p) \cong \Z_2$ is onto} \}. 
\]
It's known that 
the above definition does not depend on the choice of $p$. 
See pp.604 in \cite{FHW}. 

\section{Loop space homology} 

In this section, we recall Morse theory on loop spaces for Lagrangian action functionals. 
We mainly follow \cite{AbM}, \cite{AbS2}.

\subsection{Lagrangian action functional}
Recall that we used the notation $\Lambda(\R^n):= W^{1,2}(S^1,\R^n)$. 
Given $L \in C^\infty(S^1 \times T\R^n)$, we consider the action functional 
\[
\mca{S}_L:\  \Lambda(\R^n) \to \R; \quad \gamma \mapsto \int_{S^1} L(t,\gamma(t), \dot{\gamma}(t))\, dt.
\]
We introduce the following conditions on $L$: 
\begin{enumerate}
\item[(L1):]
There exists $a \in (0,\infty) \setminus \pi \Z$ such that 
\[
\sup_{t \in S^1}\bigg\| L(t,q,v) - \biggl( \frac{|v|^2}{4a} - a|q|^2 \biggr) \bigg\|_{C^2(T\R^n)} < \infty.
\] 
\item[(L2):]
There exists a constant $c>0$ such that 
$\partial_v^2 L(t,q,v) \ge c$ 
for any $(t,q,v) \in S^1 \times T\R^n$. 
\end{enumerate}
Notice that (L1) implies the following estimates: 
\begin{align*}
\text{(L1)'}:\quad &| D^2L(t,q,v)| \le \const, \\
&|\partial_qL(t,q,v)| \le \const (1+|q|), \quad |\partial_vL(t,q,v)| \le \const (1+|v|), \\
&|L(t,q,v)| \le \const (1+|q|^2+|v|^2).
\end{align*}

\begin{lem}\label{lem:dSL}
If $L$ satisfies (L1) and (L2), the following holds. 
\begin{enumerate}
\item 
$\mca{S}_L: \Lambda(\R^n) \to \R$ is a Fr\'{e}ch\'{e}t $C^1$ function. 
Its differential $d\mca{S}_L$ is given by 
\[
d\mca{S}_L(\gamma)(\xi)= \int_{S^1} \partial_q L(t,\gamma, \dot{\gamma}) \xi(t) + \partial_v L(t,\gamma, \dot{\gamma}) \dot{\xi}(t) \, dt. 
\]
Moreover, $d\mca{S}_L$ is Gataux differentiable. We denote the differential by $d^2\mca{S}_L$. 
\item
$\gamma \in \Lambda(\R^n)$ satisfies $d\mca{S}_L(\gamma)=0$ if and only if 
$\gamma \in C^\infty(S^1,\R^n)$ and 
\[
\partial_q L(t,\gamma, \dot{\gamma}) - \frac{d}{dt} \bigl( \partial_v L(t,\gamma, \dot{\gamma}) \bigr) = 0.
\]
\end{enumerate}
\end{lem}
\begin{proof}
Using (L1)' and (L2), the proof is the same as Proposition 3.1 in \cite{AbS2}.  
\end{proof}

Let us set $\mca{P}(L):= \{ \gamma \in \Lambda(\R^n) \mid d\mca{S}_L(\gamma)=0 \}$. 
$\gamma \in \mca{P}(L)$ is called \textit{nondegenerate} if 
$d^2\mca{S}_L(\gamma)$ is nondegenerate as a symmetric bilinear form on $T_{\gamma}\Lambda(\R^n) = W^{1,2}(S^1,\R^n)$. 

For each $\gamma \in \Lambda(\R^n)$, $D\mca{S}_L(\gamma) \in T_{\gamma}\Lambda(\R^n)=W^{1,2}(S^1,\R^n)$ is defined so that 
\[
\langle D\mca{S}_L(\gamma), \xi \rangle_{W^{1,2}} = d\mca{S}_L(\gamma)(\xi) \qquad \bigl(\forall \xi \in W^{1,2}(S^1,\R^n) \bigr).
\]
We show that the pair $(\mca{S}_L, D \mca{S}_L)$ satisfies the Palais-Smale (PS) condition. 
First let us recall what the PS condition is:

\begin{defn}
Let $M$ be a Hilbert manifold, $f:M \to \R$ be a $C^1$ function, and $X$ be a continuous vector field on $M$. 
A sequence $(p_k)_k$ on $M$ is called a \textit{Palais-Smale (PS) sequence}, if $(f(p_k))_k$ is bounded, 
and $\lim_{k \to \infty} df(X(p_k))=0$. 
The pair $(f,X)$ satisfies the \textit{PS-condition}, if any PS sequence contains a convergent subsequence.
\end{defn}

\begin{lem}\label{lem:PS}
Suppose that $L \in C^\infty(S^1 \times T\R^n)$ satisfies (L1).
Let $(\gamma_k)_k$ be a sequence on $\Lambda(\R^n)$ such that both 
$\mca{S}_L(\gamma_k)$ and $\| D \mca{S}_L(\gamma_k) \|_{W^{1,2}}$ are bounded.  
Then, $(\gamma_k)_k$ is $C^0$-bounded. 
\end{lem} 
\begin{proof}
Suppose that there exists a sequence $(\gamma_k)_k$ such that 
both $\mca{S}_L(\gamma_k)$, $\| D\mca{S}_L(\gamma_k) \|_{W^{1,2}}$ are bounded, and 
$m_k:= \max_{t \in S^1} |\gamma_k(t)|$ goes to $\infty$ as $k \to \infty$. 
We define $\delta_k \in \Lambda(\R^n)$ and $l_k \in C^\infty(S^1 \times T\R^n)$ by 
\[
\delta_k(t):= \gamma_k(t)/m_k, \qquad
l_k(t,q,p):= L(t,m_kq, m_kp)/{m_k}^2. 
\]

We show that $(\delta_k)_k$ is $W^{1,2}$-bounded.
Since $(\delta_k)_k$ is obviously $C^0$-bounded, it is enough to show that 
$(\dot{\delta_k})_k$ is $L^2$-bounded. 
First notice that
\[
\lim_{k \to \infty} \mca{S}_{l_k}(\delta_k) = \lim_{k \to \infty} \frac{\mca{S}_L(\gamma_k)}{{m_k}^2} =0. 
\]
On the other hand, since $L$ satisfies (L1), 
\[
\lim_{k \to \infty} 
\mca{S}_{l_k}(\delta_k)  - \biggl( \int_{S^1}  \frac{|\dot{\delta_k}|^2}{4a} - a |\delta_k|^2 \, dt \biggr) = 0. 
\]
Thus $(\dot{\delta_k})_k$ is $L^2$-bounded. 

By taking a subsequence of $(\delta_k)_k$, we may assume that 
there exists 
$\delta \in \Lambda(\R^n)$ such that 
$\lim_{k \to \infty} \|\delta_k - \delta \|_{C^0}=0$, and 
$\dot{\delta_k} \to \dot{\delta} \,(k \to \infty)$ weakly in $L^2$. 

We prove that $d \mca{S}_l(\delta)=0$, where
$l(t,q,v):=|v|^2/4a -a |q|^2$. This means that $\delta \in C^\infty(S^1,\R^n)$ and 
$\ddot{\delta}(t) + 4a^2 \delta(t) \equiv 0$. 
Since $a \notin \pi \Z$, this means that $\delta(t) \equiv 0$.
However, since $\max_{t \in S^1} |\delta(t)| = \lim_{k \to \infty} \max_{t \in S^1} |\delta_k(t)| =1$, 
this is a contradiction.

To prove $d\mca{S}_l(\delta)=0$, first notice that 
\[
\lim_{k \to \infty}  \| D\mca{S}_{l_k}(\delta_k) \|_{W^{1,2}} = \lim_{k \to \infty} \frac{\| D\mca{S}_L(\gamma_k) \|_{W^{1,2}}}{m_k} =0.
\]
Hence it is enough to show that for any $\xi \in C^\infty(S^1, \R^n)$ there holds 
\[
\lim_{k \to \infty} \bigl( d\mca{S}_l(\delta) - d \mca{S}_l(\delta_k) \bigr) (\xi) =0, \qquad 
\lim_{k \to \infty} \bigl( d\mca{S}_l(\delta_k) - d\mca{S}_{l_k}(\delta_k) \bigr) (\xi) =0. 
\]
To check the first claim, notice the following equation:
\[
\bigl( d\mca{S}_l(\delta) - d\mca{S}_l(\delta_k) \bigr) (\xi) 
=\int_{S^1} \frac{\dot{\delta}(t) - \dot{\delta_k}(t)}{2a}  \cdot \dot{\xi}(t) - 2a \bigl( \delta(t) - \delta_k(t) \bigr) \cdot \xi(t) \, dt. 
\]
Then, since $\dot{\delta_k}$ converges to $\dot{\delta}$ weakly in $L^2$, the RHS goes to $0$ as $k \to \infty$. 
The second claim follows from 
$\lim_{k \to \infty} \| l-l_k \|_{C^1}=0$. 
\end{proof}

\begin{cor}\label{cor:PS}
Suppose that $L \in C^\infty(S^1 \times T\R^n)$ satisfies (L1) and (L2). 
Then, the pair $(\mca{S}_L, D\mca{S}_L)$ satisfies the PS-condition on $\Lambda(\R^n)$.
\end{cor}
\begin{proof}
Suppose that $(\gamma_k)_k$ is a PS-sequence with respect to $(\mca{S}_L, D\mca{S}_L)$. 
Then, Lemma \ref{lem:PS} shows that $(\gamma_k)_k$ is $C^0$- bounded. 
Then, Proposition 3.3 in \cite{AbS2} shows that $(\gamma_k)_k$ has a convergent subsequence. 
\end{proof} 

\subsection{Construction of a downward pseudo-gradient} 
Suppose that $L \in C^\infty(S^1 \times T\R^n)$ satisfies (L1) and (L2). 
To define a Morse complex of $\mca{S}_L$, we need the following condition: 
\begin{enumerate}
\item[(L0):]
Every $\gamma \in \mca{P}(L)$ is nondegenerate. 
\end{enumerate}

The following lemma (basically the same as Theorem 4.1 in \cite{AbS2})
constructs a downward pseudo-gradient vector field for $\mca{S}_L$. 
For the definitions of the terms "Lyapunov function", "Morse vector field", "Morse-Smale condition", see Section 2 of \cite{AbS2}. 

\begin{lem}\label{lem:vectorfield}
If $L \in C^\infty(S^1 \times T\R^n)$ satisfies (L0), (L1), (L2), 
there exists a smooth vector field $X$ on 
$\Lambda(\R^n)$ which satisfies the following conditions: 
\begin{enumerate}
\item $X$ is complete.
\item $\mca{S}_L$ is a Lyapunov function for $X$. 
\item $X$ is a Morse vector field. $X(\gamma)=0$ if and only if $\gamma \in \mca{P}(L)$. Every $\gamma \in \mca{P}(L)$ has a finite Morse index, 
which is denoted by $\ind_\Morse(\gamma)$. 
\item The pair $(\mca{S}_L, X)$ satisfies the Palais-Smale condition. 
\item $X$ satisfies the Morse-Smale condition up to every order. 
\end{enumerate}
\end{lem}
\begin{proof}
In the course of this proof, we use the following abbreviation:
\[
\{a < \mca{S}_L < b\}:= \{ \gamma \in \Lambda(\R^n) \mid a< \mca{S}_L(\gamma) < b\}.
\] 
Moreover, $\|\, \cdot \, \|_{W^{1,2}}$ is abbreviated as $\| \, \cdot \, \|$. 

Since $(\mca{S}_L, D \mca{S}_L)$ satisfies the PS-condition, and all critical points are nondegenerate, 
for any $a<b$ there exist only finitely many critical points of $\mca{S}_L$ on 
$\{a <\mca{S}_L<b\}$. 
We denote them as $\gamma_1, \ldots, \gamma_m$. 

For each $1 \le j \le m$, Lemma 4.1 in \cite{AbS2} shows that there exist $U_{\gamma_j}$, $Y_{\gamma_j}$ such that:
\begin{itemize}
\item $U_{\gamma_j}$ is a neighborhood of $\gamma_j$ in $\{a< \mca{S}_L < b\}$. 
\item $Y_{\gamma_j}$ is a smooth vector field on $U_{\gamma_j}$. 
\item $\gamma_j$ is a critical point of $Y_{\gamma_j}$ with a finite Morse index, and there holds 
\[
d\mca{S}_L(Y_{\gamma_j}(\gamma)) \le -\lambda(\gamma_j) \| \gamma- \gamma_j \|^2 \qquad (\forall \gamma \in U_{\gamma_j}), 
\]
where $\lambda(\gamma_j)$ is a positive constant.
\end{itemize}
By taking $U_{\gamma_j}$ sufficiently small, we may asume that $\|Y_{\gamma_j}\| \le 1$ on $U_{\gamma_j}$. 
We take a smaller neighborhood $V_{\gamma_j}$ such that $\overline{V_{\gamma_j}} \subset U_{\gamma_j}$. 

Since $(\mca{S}_L, D \mca{S}_L)$ satisfies the PS-condition, there exists $\varepsilon>0$ such that : 
for any $\gamma \in \{ a< \mca{S}_L < b\} \setminus (U_{\gamma_1} \cup \cdots \cup U_{\gamma_m})$, 
$\| D \mca{S}_L(\gamma) \| \ge \varepsilon$. 
For each $\gamma \notin U_{\gamma_1} \cup \cdots \cup U_{\gamma_m}$, 
set $Y_\gamma:=-D\mca{S}_L(\gamma)/ \| D\mca{S}_L(\gamma) \|$. 
Then, obviously $\| Y_\gamma \| =1$. Moreover, 
\[
d\mca{S}_L(\gamma)(Y_\gamma) = \langle D\mca{S}_L(\gamma), Y_\gamma \rangle = - \| D\mca{S}_L(\gamma) \| \le -\varepsilon.
\]
Since $\mca{S}_L$ is $C^1$ by Lemma \ref{lem:dSL} (1), 
if $U_\gamma$ is a sufficiently small neighborhood of $\gamma$, 
\[
\gamma' \in U_\gamma \implies d\mca{S}_L(\gamma')(Y_\gamma) \le -\varepsilon/2.
\]
We may also assume that $U_\gamma$ is disjoint from $\overline{V_{\gamma_1}} \cup \cdots \cup \overline{V_{\gamma_m}}$.
Moreover, since $\Lambda(\R^n)$ is paracompact, we can define a locally finite open covering $\{U_\gamma\}_{\gamma \in \Gamma}$ 
of $\{a < \mca{S}_L < b\}$ such that $\gamma_1, \ldots, \gamma_m \in \Gamma$. 

Let $\{\chi_\gamma\}_{\gamma \in \Gamma}$ be a partition of unity with respect to $\{U_\gamma\}_{\gamma \in \Gamma}$.  
Then we define a vector field $Y$ on $\{a< \mca{S}_L <b\}$ by $Y:= \sum_{\gamma \in \Gamma} \chi_\gamma Y_\gamma$. 
Since each $Y_\gamma$ satisfies $\| Y_\gamma \| \le 1$, 
it is clear that $\| Y \| \le 1$. 
Moreover, there exists $c>0$ such that
\[
\gamma \notin V_{\gamma_1} \cup \cdots \cup V_{\gamma_m} \implies 
d\mca{S}_L(\gamma)(Y(\gamma)) \le -c.
\]

Now we show that $(\mca{S}_L, Y)$ satisfies the PS-condition on $\{a< \mca{S}_L < b\}$. 
Let $(x_k)_k$ be a sequence on $\{a<\mca{S}_L<b\}$ such that $\lim_{k \to \infty} d\mca{S}_L(x_k)(Y(x_k))=0$.
Then, $x_k \in V_{\gamma_1} \cup \cdots \cup V_{\gamma_m}$ for sufficiently large $k$. 
By taking a subsequence, we may assume that $x_k \in V_{\gamma_1}$ for all $k$. Then, since
\[
d\mca{S}_L(x_k)(Y(x_k)) = d\mca{S}_L(x_k)(Y_{\gamma_1}(x_k)) \le -\lambda(\gamma_1) \| x_k - \gamma_1 \|^2,
\]
there holds $\lim_{k \to \infty} \| x_k - \gamma_1 \| = 0$. Thus $(\mca{S}_L, Y)$ satisfies the PS condition. 
We have defined a smooth vector field $Y$ on $\{a< \mca{S}_L < b\}$, which satisfies (2), (3), (4) and $\| Y \| \le 1$. 

Finally we construct $X$ on $\Lambda(\R^n)$. 
Take a sequence of closed intervals $(I_m)_{m \in \Z}$ with the following properties:
\begin{itemize}
\item $(\min I_m)_m$, $(\max I_m)_m$ are increasing sequences. 
\item $\bigcup_m I_m = \R$.
\item $I_m \cap I_{m'}  \ne \emptyset$ if and only if $|m- m'| \le 1$. 
\item For any $m \in \Z$, $I_m \cap I_{m+1}$ does not contain critical values of $\mca{S}_L$. 
\end{itemize}
For every $m$, there exists a smooth vector field $X_m$ on $\{ \min I_m < \mca{S}_L < \max I_m\}$ which satisfies (2), (3), (4) and $\| X_m \| \le 1$. 
Finally, taking a partition of unity $(\rho_m)_m$ with respect to the open covering $\{ \min I_m < \mca{S}_L < \max I_m\}_m$ of $\Lambda(\R^n)$, 
we define a vector field $X$ on $\Lambda(\R^n)$ by 
$X: = \sum_m \rho_m X_m$. 
Then, it is easy to check that $(\mca{S}_L, X)$ satisfies the PS condition. 
Moreover, since $X$ satisfies $\| X \| \le 1$ everywhere, $X$ is complete. 

The vector field $X$ defined above satisfies (1)-(4) in the statement. 
Since it is of class $C^\infty$, the Sard-Smale theorem shows that (5) is satisfied by a sufficiently small $C^\infty$ perturbation. 
\end{proof}

\subsection{Morse complex} 

Let $X$ be a downward pseudo-gradient for $\mca{S}_L$ on $\Lambda(\R^n)$, which is constructed in Lemma \ref{lem:vectorfield}. 
Since $X$ is complete, one can define $(\varphi^X_t)_{t \in \R}$, a family of diffeomorphisms on $\Lambda(\R^n)$ so that 
\[
\varphi^X_0 = \id_{\Lambda(\R^n)}, \qquad
\partial_t (\varphi^X_t) = X(\varphi^X_t). 
\]
For each $\gamma \in \mca{P}(L)$, its stable and unstable manifolds are defined as 
\begin{align*}
W^s(\gamma:X)&:= \{ p \in \Lambda(\R^n) \mid \lim_{t \to \infty} \varphi^X_t(p) = \gamma \}, \\
W^u(\gamma:X)&:= \{ p \in \Lambda(\R^n) \mid \lim_{t \to -\infty} \varphi^X_t(p) = \gamma \}.
\end{align*}
For any $\gamma, \gamma' \in \mca{P}(L)$, 
we set $\mca{M}_X(\gamma, \gamma'):= W^u(\gamma:X) \cap W^s(\gamma',X)$. 
Since $\mca{M}_X(\gamma, \gamma')$ consists of flow lines of $X$, $\mca{M}_X(\gamma,\gamma')$ admits a natural $\R$ action.  
We denote the quotient by $\bar{\mca{M}}_X(\gamma,\gamma')$. 

For any $\gamma, \gamma' \in \mca{P}(L)$, $W^u(\gamma:X)$ and $W^s(\gamma':X)$ are transverse, since $X$ satisfies the Morse-Smale condition. 
Therefore, $\bar{\mca{M}}_X(\gamma,\gamma')$ is a smooth manifold with dimension $\ind_\Morse(\gamma) - \ind_\Morse(\gamma')-1$. 
When $\ind_\Morse(\gamma) - \ind_\Morse(\gamma') =1$, $\bar{\mca{M}}_X(\gamma, \gamma')$ consists of finitely many points. 

For any $-\infty < a < b < \infty$, $\CM_*^{[a,b)}(L)$ denotes the free $\Z_2$-module generated by
$\{ \gamma \in \mca{P}(L) \mid a \le \mca{S}_L(\gamma) <  b\}$. 
We define a differential $\partial_{L,X}$ on $\CM_*^{[a,b)}(L)$ by 
\[
\partial_{L,X}\bigl([\gamma]\bigr):= \sum_{\ind_\Morse(\gamma')=\ind_\Morse(\gamma)-1} \sharp \bar{\mca{M}}_X(\gamma,\gamma') \cdot [\gamma'].
\]
Then $(\CM_*^{[a,b)}(L), \partial_{L,X})$ is a chain complex, and its homology group
$\HM_*^{[a,b)}(L,X)$ is isomorphic to $H_*(\{\mca{S}_L<b\}, \{\mca{S}_L<a\})$. 
For details, see \cite{AbM}.

Next we discuss functoriality. 
Consider $L^0, L^1 \in C^\infty(S^1 \times T\R^n)$ which satisfy (L0), (L1), (L2) and 
$L^0(t,q,v)>L^1(t,q,v)$ for any $(t,q,v) \in S^1 \times T\R^n$. 
Take vector fields $X^0, X^1$ on $\Lambda(\R^n)$ such that 
$(L^0,X^0)$ and $(L^1,X^1)$ satisfy the conditions in Lemma \ref{lem:vectorfield}. 

We assume that $\mca{P}(L^0) \cap \mca{P}(L^1)= \emptyset$ (this can be achieved by slightly perturbing $L^0$).
Then, by a $C^\infty$-small perturbation of $X^0$, one can assume the following: 
\begin{quote}
For any $\gamma^0 \in \mca{P}(L^0)$ and $\gamma^1 \in \mca{P}(L^1)$, 
$W^u(\gamma^0:X^0)$ is transverse to $W^s(\gamma^1 : X^1)$. 
\end{quote}
If this assumption is satisfied, 
$\mca{M}_{X^0, X^1}(\gamma^0, \gamma^1):= W^u(\gamma^0:X^0) \cap W^s(\gamma^1:X^1)$
is a smooth manifold with dimension $\ind_\Morse(\gamma^0) - \ind_\Morse(\gamma^1)$. 

We define a chain map 
$\Phi: \CM_*^{[a,b)}(L^0,X^0) \to \CM_*^{[a,b)}(L^1,X^1)$ by
\[
\Phi([\gamma]):= \sum_{\ind_\Morse(\gamma')=\ind_\Morse(\gamma)} \sharp 
\mca{M}_{X^0, X^1}(\gamma, \gamma')\cdot [\gamma'].
\]
$\Phi$ induces a homomorphism on homology, which coincides with the homomorphism induced by the inclusion 
$(\{\mca{S}_{L^0}<a\}, \{ \mca{S}_{L^0} <b \}) \to (\{\mca{S}_{L^1}<a\}, \{ \mca{S}_{L^1} <b \})$.

\section{Proof of Theorem \ref{thm:main}}

The goal of this section is to prove Theorem \ref{thm:main}, i.e. to compute $\SH_*^{[a,b)}(D^*V) $ for a bounded domain $V \subset \R^n$ with smooth boundary. 
In Section 4.1, we reduce Theorem \ref{thm:main} to Theorem \ref{thm:key} and Lemma \ref{lem:excision}. 
Theorem \ref{thm:key} is the main step, and it is proved in Sections 4.2 and 4.3, assuming some $C^0$- estimates of Floer trajectories: 
Lemmas \ref{lem:C0_3}, \ref{lem:C0_4}, \ref{lem:C0_5}. 
These $C^0$- estimates are proved in Section 5. 
Lemma \ref{lem:excision} is a technical lemma on loop space homology,
and it is proved in Section 4.4. 

\subsection{Outline}
Let us take $(a_m)_m$, an increasing sequence of positive numbers such that $a_m \notin \pi \Z$ for any $m$, and $\lim_{m \to \infty} a_m = \infty$. 
We take a sequence $(k_m)_m$ in $C^\infty(\R_{\ge 0},\R)$ such that:
\begin{enumerate}
\item[(k1):] For every $m$, $\partial_t k_m(t)>0$ and $\partial_t^2 k_m(t) \ge 0$ for any $t \ge 0$. 
\item[(k2):] For every $m$, $\partial_t k_m \equiv a_m$ on $\{ t \mid k_m(t) \ge 0\}$.  
\item[(k3):] $(k_m)_m$ is strictly increasing. Moreover, $\sup_m k_m(t) = \begin{cases} 0 &(0 \le t \le 1) \\  \infty &(t >1) \end{cases}$.
\end{enumerate}
Let us define $K_m \in C^\infty(\R^n,\R)$ by 
$K_m(p):= k_m(|p|^2)$. 
Then, (k1) implies that $K_m$ is strictly convex. 
Moreover, (k2) implies that 
\[
\R^n \to \R^n; p \mapsto v(p):= \partial_p K_m
\]
is a diffeomorphism. We denote its inverse by $p(v)$, i.e.
$\partial_p K_m(p(v)) = v$. 
Let $K^{\vee}_m$ be the Legendre transform of $K_m$, i.e. 
$K^{\vee}_m(v):= p(v) \cdot v - K_m(p(v))$. 
Then, it is easy to show that $(K^{\vee}_m)_m$ is strictly decreasing, and 
$\inf_m K^{\vee}_m(v) = |v|$ for any $v \in \R^n$. 

We take a sequence $(Q_m)_m$ of smooth functions on $\R^n$, such that 
\begin{enumerate} 
\item[(Q1):] There exists a sequence of constants $(c_m)_m$ such that $Q_m(q) - \bigl( a_m |q|^2 + c_m\bigr)$ is compactly supported. 
\item[(Q2):] $(Q_m)_m$ is strictly increasing. Moreover, 
$\sup_m Q_m(q) = \begin{cases} 0 &( q \in \bar{V}) \\ \infty &( q \notin \bar{V}) \end{cases}$. 
\end{enumerate}

Let $H'_m(q,p):=Q_m(q) + K_m(p)$. 
Then, for every $m$, $H'_m$ satisfies (H1). 
Moreover, $(H'_m)_m$ is strictly increasing, and 
\[
\sup_m H'_m(q,p) =  \begin{cases} 0 &\bigl( (q,p) \in \overline{D^*V} \bigr) \\ \infty &\bigl( (q,p) \notin \overline{D^*V} \bigr) \end{cases}.
\]
Let $L'_m$ be the fiberwise Legendre transform of $H'_m$. 
It is easy to see that $L'_m(q,v)=K^{\vee}_m(v) - Q_m(q)$. 
Then, for every $m$, $L'_m$ satisfies (L1) and (L2). 
$(L'_m)_m$ is strictly decreasing, and there holds 
$\inf_m L'_m(q,v) = \begin{cases} |v| &(q \in \bar{V}) \\ -\infty &(q \notin \bar{V}) \end{cases}$. 

Since $(H'_m)_m$ is \textit{strictly} increasing, by sufficiently small perturbations of $(H'_m)_m$, one can obtain a sequence 
$(H^m)_m$ on $C^\infty(S^1 \times T^*\R^n)$ with the following properties:

\begin{itemize}
\item For every $m$, $H^m$ is admissible. 
\item $(H^m)_m$ is strictly increasing, and $\sup_m H^m(t,q,p) = \begin{cases} 0 &((q,p) \in \overline{D^*V}) \\ \infty &((q,p) \notin \overline{D^*V}) \end{cases}$.
\item For every $m$, its Legendre transform $L^m$ is well-defined, and it satisfies (L0), (L1), (L2). 
$(L^m)_m$ is strictly decreasing, and $\inf_m L^m(t,q,v) = \begin{cases} |v| &( q \in \bar{V}) \\ -\infty &(q \notin \bar{V}) \end{cases}$. 
\end{itemize} 

\begin{rem}
For notational reasons, we use superscripts for $H^m$ and $L^m$. 
\end{rem} 

By the first two properties, $\SH_*^{[a,b)}(D^*V) = \varinjlim_{m \to \infty} \HF^{[a,b)}_*(H^m)$. 
Now we state the following key result, which is proved in Sections 4.2 and 4.3: 

\begin{thm}\label{thm:key}
For any $-\infty < a < b < \infty$ and $m$, there exists a natural isomorphism
$\HM_*^{[a,b)}(L^m) \cong \HF_*^{[a,b)}(H^m)$. 
The following diagram is commutative for every $m$: 
\[
\xymatrix{
\HM_*^{[a,b)}(L^m) \ar[r] \ar[d]_-{\cong} & \HM_*^{[a,b)}(L^{m+1}) \ar[d]^-{\cong} \\
\HF_*^{[a,b)}(H^m) \ar[r] & \HF_*^{[a,b)}(H^{m+1}).
}
\]
\end{thm}
Then we obtain 
\[
\varinjlim_{m \to \infty} \HF_*^{[a,b)}(H^m) \cong 
\varinjlim_{m \to \infty} \HM_*^{[a,b)}(L^m) \cong 
H_* \biggl( \bigcup_m \{ \mca{S}_{L^m} < b\}, \bigcup_n \{ \mca{S}_{L^m} < a \} \biggr).
\]
Since $(L^m)_m$ is strictly decreasing and 
$\inf_m L^m(t,q,v) = \begin{cases} |v| &(q \in \bar{V}) \\ -\infty &( q \notin \bar{V}) \end{cases}$, for any $c \in \R$ 
\[
\inf_m \mca{S}_{L^m}(\gamma) < c \iff \text{$\gamma(S^1) \not\subset \bar{V}$ or $(\text{length of $\gamma$})<c$}. 
\]
Therefore, for any $a<0$ and $b>0$, 
\begin{align*}
\SH_*^{[a,b)}(D^*V) &\cong H_* \bigl( \Lambda^{<b}(\R^n) \cup (\Lambda(\R^n) \setminus \Lambda(\bar{V})), \Lambda(\R^n) \setminus \Lambda(\bar{V}) \bigr) \\
 & \cong H_* \bigl(\Lambda^{<b}(\R^n), \Lambda^{<b}(\R^n) \setminus \Lambda(\bar{V}) \bigr) \\
 & \cong H_*(\Lambda^{<b}(\bar{V}), \Lambda^{<b}(\bar{V}) \setminus \Lambda(V)), 
\end{align*}
where the second isomorphism follows from excision, and 
the third isomorphism follows from the next Lemma \ref{lem:excision}, which is proved in Section 4.4. 

\begin{lem}\label{lem:excision}
Let $V$ be a bounded domain in $\R^n$ with smooth boundary. 
For any $0 < b < \infty$, there exists a natural isomorphism 
\[
H_*(\Lambda^{<b}(\R^n), \Lambda^{<b}(\R^n) \setminus \Lambda(\bar{V}) ) \cong 
H_*(\Lambda^{<b}(\bar{V}), \Lambda^{<b}(\bar{V}) \setminus \Lambda(V)). 
\]
\end{lem}

Finally, we have to check that for any $b_- < b_+$, the following diagram commutes: 
\[
\xymatrix{
\SH_*^{[a,b^-)}(D^*V) \ar[r]^-{\cong} \ar[d] &  H_*\bigl(\Lambda^{<b^-}(\bar{V}), \Lambda^{<b^-}(\bar{V}) \setminus \Lambda(V) \bigr)\ar[d] \\
\SH_*^{[a,b^+)}(D^*V) \ar[r]_-{\cong}       & H_*\bigl(\Lambda^{<b^+}(\bar{V}), \Lambda^{<b^+}(\bar{V}) \setminus \Lambda(V) \bigr).
}
\]
This is clear from the construction, hence omitted.

\subsection{Construction of a chain level isomorphism}

In this and the next subsection, we prove Theorem \ref{thm:key}. 
In this subsection, we define an isomorphism 
\[
\HM^{[a,b)}_*(L^m) \to \HF^{[a,b)}_*(H^m). 
\]
Following \cite{AbS1}, we define
this isomorphism by considering so called \textit{hybrid moduli spaces}. 
Suppose we are given the following data: 
\begin{itemize}
\item $J^m=(J^m_t)_{t \in S^1}$, which is sufficiently close to the standard one, and $\CF_*^{[a,b)}(H^m,J^m)$ is well-defined. 
\item Smooth vector field $X^m$ on $\Lambda(\R^n)$, such that $\CM_*^{[a,b)}(L^m,X^m)$ is well-defined. 
\item $\gamma \in \mca{P}(L^m)$ and $x \in \mca{P}(H^m)$. 
\end{itemize} 
We consider the following equation for 
$u \in W^{1,3}( S^1 \times [0,\infty), T^*\R^n)$:
\begin{align*}
& \partial_s u - J^m_t\bigl( \partial_t u - X_{H^m_t}(u) \bigr) =0,  \\
& \pi(u(0)) \in W^u(\gamma:X^m), \\
& \lim_{s \to \infty} u(s) = x.
\end{align*}
$\pi$ denotes the natural projection 
$T^*\R^n \to \R^n; (q,p) \mapsto q$. 
The moduli space of solutions of this equation is denoted by $\mca{M}_{X^m,H^m,J^m}(\gamma,x)$. 

\begin{rem}
In the definition of $\mca{M}_{X^m,H^m,J^m}$, we have used a Sobolev space $W^{1,3}(S^1 \times [0,\infty), T^*\R^n)$. One can replace it with 
$W^{1,r}(S^1 \times [0,\infty), T^*\R^n)$ for any $2<r<4$. 
The condition $2<r<4$ is necessary to carry out Fredholm theory and prove $C^0$-estimates for Floer trajectories. 
\end{rem} 

To define a homomorphism by counting $\mca{M}_{X^m,H^m,J^m}(\gamma,x)$, we need the following results: 

\begin{lem}\label{lem:Fredholm}
For generic $J^m$, $\mca{M}_{X^m,H^m,J^m}(\gamma,x)$ is a smooth manifold 
of dimension $\ind_{\Morse}(\gamma) - \ind_{\CZ}(x)$ for any $\gamma \in \mca{P}(L^m)$ and $x \in \mca{P}(H^m)$.
\end{lem}
\begin{proof}
See Section 3.1 in \cite{AbS1}. 
\end{proof}

\begin{lem}\label{lem:Action}
For any $\gamma \in \mca{P}(L^m)$, $x \in \mca{P}(H^m)$ and 
$u \in \mca{M}_{X^m,H^m,J^m}(\gamma,x)$, there holds
\[
\mca{S}_{L^m}(\gamma) \ge \mca{S}_{L^m}(\pi(u(0))) \ge \mca{A}_{H^m}(u(0)) \ge \mca{A}_{H^m}(x).
\]
\end{lem}
\begin{proof}
See pp.299 in \cite{AbS1}. 
\end{proof}

\begin{cor}\label{cor:Action}
When $\mca{S}_{L^m}(\gamma) < \mca{A}_{H^m}(x)$, $\mca{M}_{X^m,H^m,J^m}(\gamma,x)=\emptyset$. 
When $\mca{S}_{L^m}(\gamma) = \mca{A}_{H^m}(x)$, $\mca{M}_{X^m,H^m,J^m}(\gamma,x) \ne \emptyset$ if and only if $\gamma = \pi(x)$. 
In this case, $\mca{M}_{X^m,H^m,J^m}(\gamma,x)$ consists of a single element $u$ such that $u(s,t):=x(t)$. 
\end{cor}

We recall that our setup differs from the one of \cite{AbS1} inasmuch as our base manifold is $\R^n$, while the authors of \cite{AbS1} work with compact bases. 
However, their analysis applies to our situation for all aspects except for the $C^0$-bounds of Floer moduli spaces. 

Now, we state our third $C^0$ -estimate. It is proved in Section 5. 

\begin{lem}\label{lem:C0_3}
There exists $\varepsilon >0$ such that, if $J^m$ satisfies 
$\sup_t \| J^m_t - J_{\std} \|_{C^0} < \varepsilon$, 
$\mca{M}_{X^m,H^m,J^m}(\gamma,x)$ is $C^0$-bounded for any 
$\gamma \in \mca{P}(L^m)$ and $x \in \mca{P}(H^m)$. 
\end{lem}

Suppose that $J^m$ satisfies the condition in Lemma \ref{lem:Fredholm}, and it is sufficiently close to $J_\std$. 
By Lemma \ref{lem:C0_3}, for any $\gamma \in \mca{P}(L^m)$ and $x \in \mca{P}(H^m)$ such that 
$\ind_{\Morse}(\gamma) - \ind_{\CZ}(x)=0$, 
$\mca{M}_{X^m,H^m,J^m}(\gamma,x)$ is a \textit{compact} $0$-dimensional manifold. 
Then, we can define a homomorphism 
\begin{align*}
&\Psi^m: \CM_*^{[a,b)}(L^m,X^m) \to \CF_*^{[a,b)}(H^m,J^m); \\
&[\gamma] \mapsto \sum_{\ind_{\CZ}(x) = \ind_{\Morse}(\gamma)} \sharp \mca{M}_{X^m,H^m,J^m}(\gamma,x) \cdot [x].
\end{align*}
Corollary \ref{cor:Action} shows that $\Psi^m$ is an isomorphism (for details, see Section 3.5 in \cite{AbS1}). 
Gluing arguments show that $\Psi^m$ is a chain map (for details, see Section 3.5 in \cite{AbS1}). 
Hence $\Psi^m$ induces an isomorphism on homology.

\subsection{Chain level commutativity up to homotopy}
In the previous subsection, we constructed a chain level isomorphism
\[
\Psi^m: \CM_*^{[a,b)}(L^m,X^m) \to \CF_*^{[a,b)}(H^m,J^m)
\]
for every $m$. In this subsection, we show that 
\[
\xymatrix{
\CM_*^{[a,b)}(L^m,X^m) \ar[r]^{\Psi^m}\ar[d]^{\Phi^L} & \CF_*^{[a,b)}(H^m,J^m) \ar[d]^{\Phi^H}\\
\CM_*^{[a,b)}(L^{m+1},X^{m+1}) \ar[r]_{\Psi^{m+1}}& \CF_*^{[a,b)}(H^{m+1},J^{m+1}) \\
}
\]
commutes up to chain homotopy, where $\Phi^H$ and $\Phi^L$ are chain maps constructed in 
Section 2.3 and Section 3.3, respectively. 

To prove this, we introduce a chain map 
\begin{align*}
\Theta&: \CM_*^{[a,b)}(L^m,X^m) \to \CF_{*}^{[a,b)}(H^{m+1},J^{m+1}); \\
&[\gamma] \mapsto \sum_{\ind_{\Morse}(\gamma)=\ind_{\CZ}(x)} \sharp \mca{M}_{X^m,H^{m+1}, J^{m+1}}(\gamma, x) \cdot [x].
\end{align*}
It is enough to show $\Phi^H \circ \Psi^m \sim \Theta \sim \Psi^{m+1} \circ \Phi^L$.
($\sim$ means chain homotopic.)

First we show that $\Psi^{m+1} \circ \Phi^L \sim \Theta$. 
For any $\gamma \in \mca{P}(L^m)$ and $x \in \mca{P}(H^{m+1})$, 
$\mca{N}^0(\gamma,x)$ denotes the set of $(\alpha,u,v)$, where 
\[
\alpha \in [0,\infty), \quad
u: [0,\alpha] \to \Lambda(\R^n), \quad 
v \in W^{1,3}([0,\infty) \times S^1, T^*\R^n)
\]
which satisfy the following conditions: 
\begin{align*}
&u(0) \in W^u(\gamma:X^m), \qquad  u(s) = \varphi^{X^{m+1}}_s(u(0))\,(0 \le s \le \alpha), \\
&\partial_s v - J^{m+1}_t(\partial_t v - X_{H^{m+1}_t}(v))=0, \quad \pi(v(0))= u(\alpha), \quad \lim_{s \to \infty} v(s)=x.
\end{align*}

We state our fourth $C^0$-estimate. It is proved in Section 5. 

\begin{lem}\label{lem:C0_4}
There exists $\varepsilon>0$ which satisfies the following property:
\begin{quote}
If $J^{m+1}$ satisfies 
$\sup_t \| J^{m+1}_t - J_{\std} \|_{C^0} < \varepsilon$, 
$\| v \|_{C^0}$ is uniformly bounded for any $(\alpha,u,v) \in \mca{N}^0(\gamma,x)$, where 
$\gamma \in \mca{P}(L^m)$ and $x \in \mca{P}(H^{m+1})$.
\end{quote}
\end{lem}
 
Suppose that $J^m$ is generic and sufficiently close to $J_\std$. Then, due to Lemma \ref{lem:C0_4} and gluing arguments, 
the following holds: 

\begin{itemize}
\item When $\ind_\Morse (\gamma) - \ind_\CZ(x)=-1$, $\mca{N}^0(\gamma,x)$ is a compact $0$-dimensional manifold. 
Every $(\alpha,u,v) \in \mca{N}^0(\gamma,x)$ satisfies $\alpha>0$. 
\item When $\ind_\Morse (\gamma) - \ind_\CZ(x)=0$, $\mca{N}^0(\gamma,x)$ is a $1$-dimensional manifold with boundary. 
Its boundary is $\{\alpha=0\}$, and its end is compactified by the following moduli spaces
(we set $k:= \ind_\Morse (\gamma) = \ind_\CZ(x)$): 
\begin{align*}
 & \bar{\mca{M}}_{X^m}(\gamma,\gamma') \times \mca{N}^0(\gamma',x)  \quad (\gamma' \in \mca{P}(L^m), \ind_\Morse(\gamma')=k-1), \\
 & \mca{M}_{X^m,X^{m+1}}(\gamma,\gamma') \times \mca{M}_{X^{m+1},H^{m+1},J^{m+1}}(\gamma',x)  \\
 &\qquad \qquad \qquad\qquad\qquad\qquad ( \gamma' \in \mca{P}(L^{m+1}), \ind_\Morse(\gamma')=k ), \\
 & \mca{N}^0(\gamma,x') \times \bar{\mca{M}}_{H^{m+1},J^{m+1}}(x',x) \quad (x' \in \mca{P}(H^{m+1}), \ind_\CZ(x') = k+1 ) . 
\end{align*}
\end{itemize}

Let us define $K^0:\CM_*^{<a}(L^m,X^m) \to \CF_{*+1}^{<a}(H^{m+1},J^{m+1})$ by
\[
K^0[\gamma]:= \sum_{\ind_{\CZ}(x)=\ind_{\Morse}(\gamma)+1} \sharp \mca{N}^0(\gamma,x) \cdot [x].
\]
Then, the above results show that 
$\partial_{H^{m+1},J^{m+1}} \circ K^0 + K^0 \circ \partial_{L^m,X^m} = \Psi^{m+1} \circ \Phi^L + \Theta$. 

Next we show that $\Phi^H \circ \Psi^m \sim \Theta$. 
Let $H \in C^\infty(\R \times S^1 \times T^*\R^n)$ be a homotopy from $H^m$ to $H^{m+1}$, and 
$J=(J_{s,t})_{(s,t) \in \R \times S^1}$ be a homotopy from $J^m$ to $J^{m+1}$.
By (HH1) and (JJ1), there exists $s_0>0$ such that
\[
(H_{s,t},J_{s,t}) = \begin{cases} (H^m_t, J^m_t) &( s \le -s_0) \\ (H^{m+1}_t, J^{m+1}_t) &( s \ge s_0) \end{cases}.
\]
 
For any $\gamma \in \mca{P}(L^m)$ and $x \in \mca{P}(H^{m+1})$, $\mca{N}^1(\gamma,x)$ denotes the set of $(\beta,w)$, where 
\[
\beta \in (-\infty, s_0], \qquad 
w \in W^{1,3}([\beta,\infty) \times S^1, T^*\R^n)
\]
which satisfy the following properties: 
\begin{align*}
\pi(w(\beta))& \in W^u(\gamma:X^m), \qquad
\partial_s w - J_{s,t}(\partial_t w - X_{H_{s,t}}(w))=0, \\
\lim_{s \to \infty}w(s)&=x.
\end{align*}

Now we state our fifth $C^0$- estimate. It is proved in Section 5. 

\begin{lem}\label{lem:C0_5}
There exists $\varepsilon >0$ which satisfies the following property:
\begin{quote}
If $J$ satisfies $\sup_{s,t} \| J_{s,t} - J_{\std} \|_{C^0} < \varepsilon$, 
$\|w \|_{C^0}$ is uniformly bounded 
for any 
$(\beta ,w) \in \mca{N}_1(\gamma,x)$, where $\gamma \in \mca{P}(L^m)$, $x \in \mca{P}(H^{m+1})$. 
\end{quote}
\end{lem}

Suppose that $J$ is generic and sufficiently close to $J_\std$. Then, by Lemma \ref{lem:C0_5} and gluing arguments, 
the following holds: 

\begin{itemize}
\item When $\ind_\Morse(\gamma) - \ind_\CZ(x) =-1$, $\mca{N}^1(\gamma,x)$ is a compact $0$-dimensional manifold. 
Every $(\beta,w) \in \mca{N}^1(\gamma,x)$ satisfies $\beta <s_0$. 
\item When $\ind_\Morse(\gamma) - \ind_\CZ(x)=0$, $\mca{N}^1(\gamma,x)$ is a $1$-dimensional manifold with boundary. 
Its boundary is $\{\beta=s_0\}$, and its ends are compactified by the following moduli spaces
(we set $k:=\ind_\Morse(\gamma) = \ind_\CZ(x)$): 
\begin{align*}
&\bar{\mca{M}}_{X^m}(\gamma,\gamma') \times \mca{N}^1(\gamma',x) \quad (\gamma' \in \mca{P}(L^m), \ind_\Morse(\gamma')=k-1), \\
&\mca{M}_{X^m,H^m,J^m}(\gamma,x') \times \mca{M}_{H,J}(x',x) \quad( x' \in \mca{P}(H^m), \ind_\CZ(x') = k), \\
&\mca{N}^1(\gamma,x') \times \bar{\mca{M}}_{H^{m+1},J^{m+1}}(x',x) \quad( x' \in \mca{P}(H^{m+1}), \ind_{\CZ}(x')=k+1). 
\end{align*} 
\end{itemize}

Let us define $K^1:\CM_*^{[a,b)}(L^m,X^m) \to \CF_{*+1}^{[a,b)}(H^{m+1},J^{m+1})$ by
\[
K^1[\gamma]:= \sum_{\ind_{\CZ}(x)=\ind_{\Morse}(\gamma)+1} \sharp \mca{N}^1(\gamma,x) \cdot [x].
\]
Then, the above results show that 
$\partial_{H^{m+1},J^{m+1}} \circ K^1 + K^1 \circ \partial_{L^m,X^m} = \Theta + \Phi^H \circ \Psi^m$.

\subsection{Proof of Lemma \ref{lem:excision}}

Finally, we prove Lemma \ref{lem:excision}. 
Through this section, $V$ denotes a bounded domain in $\R^n$ with smooth boundary. 
First we need the following lemma: 

\begin{lem}\label{lem:excision-1}
For any open neighborhood $W$ of $\bar{V}$ and 
$b>0$, the natural homomorphism 
\[
H_*(\Lambda^{<b}(W), \Lambda^{<b}(W) \setminus \Lambda(\bar{V})) \to
H_*(\Lambda^{<b}(W), \Lambda^{<b}(W) \setminus \Lambda(V))
\]
is an isomorphism. 
\end{lem}
\begin{proof}
This is equivalent to showing that 
$H_*(\Lambda^{<b}(W) \setminus \Lambda(V), \Lambda^{<b}(W) \setminus \Lambda(\bar{V}))=0$. 

Let us take a $k$-dimensional singular chain 
$\alpha=\sum_i  c_i \alpha_i  \in C_k(\Lambda^{<b}(W) \setminus \Lambda(V))$
($c_i \in \Z_2$, $\alpha_i: \Delta^k \to \Lambda^{<b}(W) \setminus \Lambda(V)$ are continuous maps) such that 
$\partial \alpha \in C_{k-1}(\Lambda^{<b}(W) \setminus \Lambda(\bar{V}))$. 
Since $\Delta^k$ is compact, there exists $b'<b$ such that $\alpha_i(\Delta^k) \subset \Lambda^{<b'}(W)$ for all $i$. 

Let us take a compactly supported smooth vector field $Z$ on $W$, which points outwards on 
$\partial V$. Let $(\varphi^Z_t)_{t \in \R}$ be the isotopy on $W$ generated by $Z$, i.e. 
$\varphi^Z_0=\id_W$, $\partial_t \varphi^Z_t = Z(\varphi^Z_t)$. 
Take $\delta>0$ and define $\alpha^t_i: \Delta^k \to \Lambda(W)$ by 
$\alpha^t_i(p):= \varphi^Z_{\delta t} \circ \alpha_i(p)\,(\forall p \in \Delta^k)$. 
When $\delta>0$ is sufficiently small, $\alpha^t_i(\Delta^k) \subset \Lambda^{<b}(W) $ for any $i$ and $0 \le t \le 1$. 

It is easy to see that $\alpha^t_i$ satisfies the following properties for any $i$ and $0 \le t \le 1$:
\begin{itemize}
\item $\alpha^0_i = \alpha_i$.  
\item $\alpha^t:= \sum_i c_i \alpha^t_i$ satisfies 
$\alpha^t \in C_k(\Lambda^{<b}(W) \setminus \Lambda(V))$ and 
$\partial \alpha^t \in C_{k-1}(\Lambda^{<b}(W) \setminus \Lambda(\bar{V}))$ for any $0 \le t \le 1$. 
\item $\alpha^1 \in C_k(\Lambda^{<b}(W) \setminus \Lambda(\bar{V}))$. 
\end{itemize} 
Then we obtain $[\alpha]=[\alpha^0]=[\alpha^1]=0$ in $H_k(\Lambda^{<b}(W) \setminus \Lambda(V), \Lambda^{<b}(W) \setminus \Lambda(\bar{V}))$. 
\end{proof}

\begin{cor}\label{cor:excision-1}
For any open neighborhood $W$ of $\bar{V}$, the natural homomorphism 
\[
H_*(\Lambda^{<b}(W), \Lambda^{<b}(W) \setminus \Lambda(V)) \to
H_*(\Lambda^{<b}(\R^n), \Lambda^{<b}(\R^n) \setminus \Lambda(V))
\]
is an isomorphism.
\end{cor}
\begin{proof}
Consider the following commutative diagram: 
\[
\xymatrix{
H_*(\Lambda^{<b}(W), \Lambda^{<b}(W) \setminus \Lambda(\bar{V})) \ar[r] \ar[d] &H_*(\Lambda^{<b}(\R^n), \Lambda^{<b}(\R^n) \setminus \Lambda(\bar{V}))  \ar[d] \\
H_*(\Lambda^{<b}(W), \Lambda^{<b}(W) \setminus \Lambda(V)) \ar[r] & H_*(\Lambda^{<b}(\R^n), \Lambda^{<b}(\R^n) \setminus \Lambda(V))
}
\]
Then, vertical arrows are isomorphism by Lemma \ref{lem:excision-1}, and the top arrow is an isomorphism by excision. 
Therefore the bottom arrow is an isomorphism. 
\end{proof} 

Applying Lemma \ref{lem:excision-1} with $W=\R^n$, 
\[
H_*(\Lambda^{<b}(\R^n), \Lambda^{<b}(\R^n) \setminus \Lambda(\bar{V})) \to
H_*(\Lambda^{<b}(\R^n), \Lambda^{<b}(\R^n) \setminus \Lambda(V))
\]
is an isomorphism. Hence, to prove Lemma \ref{lem:excision} it is enough to show that the natural homomorphism
\[
H_*(\Lambda^{<b}(\bar{V}), \Lambda^{<b}(\bar{V}) \setminus \Lambda(V)) \to
H_*(\Lambda^{<b}(\R^n), \Lambda^{<b}(\R^n) \setminus \Lambda(V))
\]
is an isomorphism. To show this, we need the following trick: 
take a sequence $(g^l)_l$ of Riemannian metrics on $\R^n$, with the following properties: 

\begin{enumerate}
\item[(g-1):] For any tangent vector $\xi$ on $\R^n$, $|\xi|_{g^l}$ is decreasing in $l$: $|\xi|_{g^1} > |\xi|_{g^2} > \cdots$. 
\item[(g-2):] For any tangent vector $\xi$ on $\R^n$, $\lim_{l \to \infty} |\xi|_{g^l} = |\xi|$, where $|\, \cdot \,|$ is the standard metric. 
\item[(g-3):] For any $l \ge 1$, there exists an embedding $\tau_l: \partial V \times (-\varepsilon_l, \varepsilon_l) \to \R^n$ with the following properties: 
\begin{itemize}
\item $\tau_l(x,0)=x$ for any $x \in \partial V$. 
\item $\tau_l^{-1}(V) = \partial V \times (-\varepsilon_l, 0)$. 
\item $\tau_l^*g^l$ is a product metric of $g^l|_{\partial V}$ and the standard metric on $(-\varepsilon_l, \varepsilon_l)$. 
\end{itemize}
We set $W_l:= V \cup \Im \tau_l$. 
\end{enumerate}
For each $l$ we define 
\[
\Lambda_l^{<b}(\R^n):= \biggl\{  \gamma \in \Lambda(\R^n) \biggm{|} \int_{S^1} |\dot{\gamma}(t)|_{g^l} \, dt < b  \biggr\}, \quad
\Lambda_l^{<b}(\bar{V}):=\Lambda_l^{<b}(\R^n) \cap \Lambda(\bar{V}).
\] 
By (g-1), $(\Lambda^{<b}_l(\R^n))_l$, $(\Lambda^{<b}_l(\bar{V}))_l$ are increasing sequences of open sets in 
$\Lambda^{<b}(\R^n)$, $\Lambda^{<b}(\bar{V})$. 
By (g-2), $\bigcup_l \Lambda^{<b}_l(\R^n) = \Lambda^{<b}(\R^n)$, 
$\bigcup_l \Lambda^{<b}_l(\bar{V}) = \Lambda^{<b}(\bar{V})$. Thus there holds 
\begin{align*}
H_*(\Lambda^{<b}(\bar{V}), \Lambda^{<b}(\bar{V}) \setminus \Lambda(V)) 
&=\varinjlim_{l \to \infty} H_*(\Lambda_l^{<b}(\bar{V}), \Lambda_l^{<b}(\bar{V}) \setminus \Lambda(V)), \\
H_*(\Lambda^{<b}(\R^n), \Lambda^{<b}(\R^n) \setminus \Lambda(V)) 
&=\varinjlim_{l \to \infty} H_*(\Lambda_l^{<b}(\R^n), \Lambda_l^{<b}(\R^n) \setminus \Lambda(V)).
\end{align*}

Therefore Lemma \ref{lem:excision} is reduced to the following lemma: 

\begin{lem}\label{lem:excision-2}
For any $l \ge 1$, the natural homomorphism 
\[
H_*(\Lambda_l^{<b}(\bar{V}), \Lambda_l^{<b}(\bar{V}) \setminus \Lambda(V)) \to
H_*(\Lambda_l^{<b}(\R^n), \Lambda_l^{<b}(\R^n) \setminus \Lambda(V))
\]
is an isomorphism. 
\end{lem}
\begin{proof}
Let us take $W_l \supset \bar{V}$ as in (g-3). 
Since Corollary \ref{cor:excision-1} is valid also for $g^l$, 
\[
H_*(\Lambda^{<b}_l(W_l), \Lambda^{<b}_l(W_l) \setminus \Lambda(V)) \to
H_*(\Lambda^{<b}_l(\R^n), \Lambda^{<b}_l(\R^n) \setminus \Lambda(V))
\]
is an isomorphism. Hence it is enough to show that 
\[
I: H_*(\Lambda_l^{<b}(\bar{V}), \Lambda_l^{<b}(\bar{V}) \setminus \Lambda(V)) \to
H_*(\Lambda_l^{<b}(W_l), \Lambda_l^{<b}(W_l) \setminus \Lambda(V))
\]
is an isomorphism. We check surjectivity and injectivity. 

We prove surjectivity of $I$. 
Take $\alpha = \sum_i c_i \alpha_i \in C_k(\Lambda_l^{<b}(W_l))$ such that $\partial \alpha \in C_{k-1}(\Lambda_l^{<b}(W_l) \setminus \Lambda(V))$.  
Since $\Delta^k$ is compact, there exists $b'<b$ such that
\[ 
\text{length of $\alpha_i(p)$ with respect to $g^l$} < b' \quad ( \forall i, \, \forall p \in \Delta^k).
\]
Let us take $\rho \in C^\infty((-\varepsilon_l,\varepsilon_l))$ with the following properties: 
\begin{itemize}
\item $\rho(s) \equiv 0$ on $[0, \varepsilon_l)$.
\item $0 \le \rho'(s) \le b/b'$, $-\varepsilon_l < \rho(s) \le 0$ on $(-\varepsilon_l, 0)$.
\item $\rho(s) \equiv s$ near $-\varepsilon_l$. 
\end{itemize}
Then we define a smooth map $\varphi: W_l \times [0,1] \to W_l; (x,t) \mapsto \varphi_t(x)$ such that: 
\begin{itemize}
\item If $x \notin \Im \tau_l$, $\varphi_t(x) = x$.
\item If $x = \tau_l(y,s)$, $\varphi_t(x) = \tau_l(y, (1-t)s + t\rho(s))$. 
\end{itemize}
It is easy to check the following properties of $\varphi$: 
\begin{itemize}
\item $\varphi_0 = \id_{W_l}$, $\varphi_1(W_l)=\bar{V}$. 
\item For any $0 \le t \le 1$, $\varphi_t(W_l \setminus V) \subset W_l \setminus V$, $\varphi_t(\bar{V}) = \bar{V}$. 
\item For any tangent vector $\xi$ on $W_l$ and $0 \le t \le 1$, $|d\varphi_t(\xi)|_{g^l} \le (b/b') |\xi|_{g^l}$.
\end{itemize}
We define $\alpha^t_i: \Delta^k \to \Lambda(W_l)$ by 
$\alpha^t_i(p):= \varphi_t \circ \alpha_i(p) \quad ( \forall p \in \Delta^k)$. 
By the last property of $\varphi$, $\alpha_i^t(\Delta^k) \subset \Lambda^{<b}_l(W_l)$. 
Moreover, $\alpha^t:= \sum_i c_i \alpha^t_i$ satisfies the following properties: 
\begin{itemize}
\item $\alpha^0=\alpha$. 
\item $\alpha^t \in C_k(\Lambda_l^{<b}(W_l))$, $\partial \alpha^t \in C_{k-1}(\Lambda_l^{<b}(W_l) \setminus \Lambda(V))$ for any $t \in [0,1]$. 
\item $\alpha^1 \in C_k(\Lambda_l^{<b}(\bar{V}))$, $\partial \alpha^1 \in C_{k-1}(\Lambda_l^{<b}(\bar{V}) \setminus \Lambda(V))$. 
\end{itemize}
Thus we obtain $[\alpha]=[\alpha^0]=[\alpha^1] \in \Im I$.
Hence we have proved surjectivity of $I$. 

We prove injectivity of $I$. 
Let $\alpha = \sum_i c_i \alpha_i \in C_k(\Lambda_l^{<b}(\bar{V}))$ such that 
$\partial \alpha \in C_{k-1}(\Lambda_l^{<b}(\bar{V}) \setminus \Lambda(V))$. 
We show that if $I([\alpha])=0$ then $[\alpha]=0$. 
By $I([\alpha])=0$, there exists 
$\beta = \sum_j d_j \beta_j \in C_{k+1}(\Lambda_l^{<b}(W_l))$ such that 
$\partial \beta - \alpha \in C_k( \Lambda_l^{<b}(W_l) \setminus \Lambda(V))$. 
Since $\Delta^k$, $\Delta^{k+1}$ are compact, there exists $b'<b$ such that 
\[
\text{length of $\alpha_i(p)$, $\beta_j(q)$ with respect to $g^l$} < b' \,
(\forall i, \, \forall j, \, \forall p \in \Delta^k, \, \forall q \in \Delta^{k+1}).
\]
Taking $\varphi: W_l \times [0,1] \to W_l$ as before, we set 
\[
\alpha^t_i:= \varphi_t \circ \alpha_i, \quad  \alpha^t:= \sum_i c_i \alpha^t_i, \quad
\beta^t_j:= \varphi_t \circ \beta_j, \quad \beta^t:= \sum_j d_j \beta^t_j.
\] 
Then, it is easy to confirm the following claims: 
\begin{itemize}
\item For any $0 \le t \le 1$, $\alpha^t \in C_k(\Lambda^{<b}_l(\bar{V}))$, $\partial \alpha^t \in C_{k-1}(\Lambda^{<b}_l(\bar{V}) \setminus \Lambda(V))$. 
\item $\beta^1 \in C_{k+1}(\Lambda^{<b}_l(\bar{V}))$.
\item $\partial \beta^1 - \alpha^1 \in C_k(\Lambda^{<b}_l(\bar{V}) \setminus \Lambda(V))$. 
\end{itemize}
Thus we obtain 
$[\alpha] = [\alpha^1] = [\partial \beta^1]=0$ in $H_k(\Lambda^{<b}_l(\bar{V}), \Lambda^{<b}_l(\bar{V}) \setminus \Lambda(V))$. 
Hence we have proved injectivity of $I$. 
\end{proof}

\section{$C^0$-estimates}

The goal of this section is to prove Lemmas on $C^0$-estimates for Floer trajectories: 
Lemma \ref{lem:C0_1}, \ref{lem:C0_2}, \ref{lem:C0_3}, \ref{lem:C0_4}, \ref{lem:C0_5}.

\subsection{$W^{1,2}$-estimate}

The goal of this subsection is to prove the following $W^{1,2}$- estimate. 
In the following statement, an expression "$c_0(H,M)$" means that $c_0$ is a constant which depends on $H$ and $M$. 

\begin{prop}\label{prop:L2}
For any $H \in C^\infty(\R \times S^1 \times T^*\R^n)$ satisfying (HH2), (HH3)
and $M>0$, there exists a constant $c_0(H,M)>0$ which satisfies the following property:
\begin{quote}
Let $I \subset \R$ be a closed interval of length $\le 3$, 
and $(J_{s,t})_{(s,t) \in I \times S^1}$ be a $I \times S^1$ -family of almost complex structures on $T^*\R^n$, such that 
every $J_{s,t}$ is compatible with $\omega_n$. 
Suppose that there holds 
\[
\frac{|\xi|^2}{2} \le \omega_n(\xi, J_{s,t}(\xi)) \le 2 |\xi|^2 
\]
for any $s \in I$, $t \in S^1$ and tangent vector $\xi$ on $T^*\R^n$. 
Then, for any $W^{1,3}$-map $u: I \times S^1 \to T^*\R^n$ which satisfies 
\[
\partial_s u - J_{s,t}(\partial_t u - X_{H_{s,t}}(u))=0, \qquad 
\sup_{s \in I} \big\lvert \mca{A}_{H_s}(u(s)) \big\rvert \le M, 
\]
there holds $\| u \|_{W^{1,2}(I \times S^1)} \le c_0$. 
\end{quote}
\end{prop}
\begin{rem}
$H_s \in C^\infty(S^1 \times T^*\R^n)$ is defined as $H_s(t,q,p):=H(s,t,q,p)$. 
\end{rem}

A crucial step is the following lemma. 

\begin{lem}\label{lem:L2}
Let $H$ and $I$ be as in Proposition \ref{prop:L2}. 
Then, there exists a constant $c_1(H)>0$ such that: 
for any $x \in C^\infty(S^1 , T^*\R^n)$ and $s \in I$, there holds 
\[
\| x \|_{L^2} ^2 + \| \partial_t x \|_{L^2}^2 
\le c_1 \biggl(1 + \int_{S^1} |\partial_t x - X_{H_{s,t}}(x(t)) |^2 + \partial_s H_{s,t}(x(t)) \, dt\biggr). 
\]
\end{lem}
\begin{proof}
Let us take $c_2(H)$ so that $c_2 > \sup_{s,t} \| \partial_s \Delta_{s,t} \|_{C^0}$
(recall that $\Delta_{s,t}$ was defined in (HH3)). 
Then we show that there exists a constant $c_3(H)>0$ such that there holds 
\begin{equation}\label{eq:xl2}
\| x \|_{L^2}^2 
\le c_3 \biggl(c_2 + \int_{S^1} |\partial_t x - X_{H_{s,t}}(x(t)) |^2 + \partial_s H_{s,t}(x(t)) \, dt\biggr)
\end{equation}
for any $x \in C^\infty(S^1, T^*\R^n)$ and $s \in I$. 
Suppose that this does not hold. 
Then, there exists a sequence $(x_k)_k$ and $(s_k)_k$ such that 
\begin{equation}\label{eq:x_k}
\frac{\| x_k \|_{L^2}^2}{c_2 + \int_{S^1} |\partial_t x_k - X_{H_{s_k,t}}(x_k(t))|^2 + \partial_s H_{s_k,t}(x_k(t))\, dt } \to \infty 
\qquad (k \to \infty).  
\end{equation}
Since $c_2 + \partial_s H_{s,t}(q,p)>0$ for any $(s,t,q,p)$, there also holds 
\[
\frac{ \| \partial_t x_k - X_{H_{s_k}}(x_k) \|_{L^2} }{ \|x_k\|_{L^2}}  \to 0
\qquad (k \to \infty).
\]

Let us set $m_k:= \|x_k\|_{L^2}$, and $v_k:=x_k/m_k$. 
Then, obviously $\| v_k \|_{L^2}=1$. 
We show that $(v_k)_k$ is $W^{1,2}$-bounded, i.e. $(\partial_t v_k)_k$ is $L^2$-bounded.
To show this, we set  $h^k(t,q,p):= H_{s_k,t}(m_k q, m_k p)/{m_k}^2$, and consider the inequality 
\[
\| \partial_t v_k \|_{L^2} \le \| \partial_t v_k - X_{h^k}(v_k) \|_{L^2} + \| X_{h^k} (v_k) \|_{L^2}. 
\]
$\| \partial_t v_k - X_{h^k}(v_k) \|_{L^2}$ is bounded in $k$, since
\begin{equation}\label{eq:dtvkxhkt}
\| \partial_t v_k - X_{h^k}(v_k) \|_{L^2} = 
\frac{ \| \partial_t x_k - X_{H_{s_k}}(x_k) \|_{L^2}}{m_k}  \to 0  \qquad (k \to \infty). 
\end{equation}
To bound $\| X_{h^k} (v_k) \|_{L^2}$, we use the inequality 
\begin{equation}\label{eq:xqas}
\| X_{Q^{a(s_k)}}(v_k) - X_{h^k}(v_k) \|_{L^2} \le \| X_{Q^{a(s_k)}}(v_k) - X_{h^k}(v_k) \|_{C^0} \le 
\frac{\sup_{t \in S^1} \| \Delta_{s_k,t}\|_{C^1}}{m_k}.
\end{equation}
Then, it is easy to see that there exists $c_4(H)>0$ such that 
$\| X_{h^k}(v_k) \|_{L^2} \le c_4 ( 1+ \|v_k \|_{L^2})$.  
Thus we have proved that $(v_k)_k$ is $W^{1,2}$-bounded. 

By taking a subsequence of $(v_k)_k$, we may assume that 
there exists $v \in W^{1,2}(S^1, T^*\R^n)$ such that $\lim_{k \to \infty} \|v - v_k \|_{C^0}=0$, 
and $\partial_t v_k$ converges to $\partial_t v$ weakly in $L^2$. 
Moreover, we may assume that $(s_k)_k$ converges to $s \in I$. 

We show that $\lim_{k \to \infty} \| X_{Q^{a(s)}}(v) - X_{h^k}(v_k) \|_{L^2} =0$. 
By the triangle inequality, 
\[
\| X_{Q^{a(s)}}(v) - X_{h^k}(v_k) \|_{L^2} \le 
\| X_{Q^{a(s)}}(v) - X_{Q^{a(s)}}(v_k) \|_{L^2} + 
\| X_{Q^{a(s)}}(v_k) - X_{h^k}(v_k) \|_{L^2}. 
\]
Then, $\lim_{k \to \infty} \| X_{Q^{a(s)}}(v) - X_{Q^{a(s)}}(v_k) \|_{L^2}=0$ since $\lim_{k \to \infty} \| v-v_k \|_{L^2} =0$. 
On the other hand, (\ref{eq:xqas}) shows that 
$\lim_{k \to \infty} \| X_{Q^{a(s)}}(v_k) - X_{h^k}(v_k) \|_{L^2}=0$. 

Now we  show that $\partial_t v - X_{Q^{a(s)}}(v) =0$ in $L^2(S^1, T^*\R^n)$, i.e.
\[
\langle \partial_t v - X_{Q^{a(s)}}(v), \xi \rangle_{L^2} =0
\]
for any $\xi \in C^\infty(S^1, T^*\R^n)$. This follows from 
\[
\langle \partial_t v - X_{Q^{a(s)}}(v), \xi \rangle_{L^2}
= \lim_{k \to \infty} \langle \partial_t v_k -  X_{h^k}(v_k), \xi \rangle_{L^2} = 0.  
\]
The first equality holds since in $L^2(S^1, T^*\R^n)$
\[
\text{$\partial_t v_k$ converges to $\partial_t v$ (weakly)}, \quad 
\text{$X_{h^k}(v_k)$ converges to $X_{Q^{a(s)}}(v)$ (in norm)}.
\]
The second equality follows from (\ref{eq:dtvkxhkt}). 

Now we have shown that $\partial_t v - X_{Q^{a(s)}}(v) =0$ in $L^2(S^1, T^*\R^n)$. 
Therefore, by a boot strapping argument, we conclude that $v \in C^\infty(S^1, T^*\R^n)$. 
This implies that $a(s) \in \pi\Z$, hence $a'(s)>0$ by (HH3). 
Hence we obtain 
\begin{align*}
\frac{m_k^2}{c_2 + \int_{S^1} \partial_s H_{s_k,t}(x_k(t)) \, dt}
&\le \frac{m_k^2}{\int_{S^1} Q^{a'(s_k)}(x_k(t))\, dt} \\
&= \frac{1}{\int_{S^1} Q^{a'(s_k)}(v_k(t))\, dt}
\to \frac{1}{a'(s) \|v \|_{L^2}^2}\quad(k \to \infty).
\end{align*}
However, this contradicts the assumption that $(x_k)_k$ satisfies 
(\ref{eq:x_k}). 
Hence we have proved (\ref{eq:xl2}). 
Setting $c_5:= \max\{c_2c_3, c_3\}$, there holds 
\begin{equation}\label{eq:xl2'}
\| x \|_{L^2}^2 
\le c_5 \biggl(1 + \int_{S^1} |\partial_t x - X_{H_{s,t}}(x(t)) |^2 + \partial_s H_{s,t}(x(t)) \, dt\biggr)
\end{equation}
for any $x \in C^\infty(S^1,T^*\R^n)$ and $s \in I$. 
Now, it is enough to show that there exists $c_6(H)>0$ such that 
\begin{equation}\label{eq:dxl2}
\| \partial_t x \|_{L^2} ^2 
\le c_6 \biggl(1 + \int_{S^1} |\partial_t x - X_{H_{s,t}}(x(t)) |^2 + \partial_s H_{s,t}(x(t)) \, dt\biggr).
\end{equation}
By using 
\begin{align*}
\| \partial_t x\|_{L^2} &\le \| \partial_t x - X_{H_s}(x) \|_{L^2} + \|X_{H_s}(x)\|_{L^2} \\
                        &\le \| \partial_t x - X_{H_s}(x) \|_{L^2} + 2a(s) \|x\|_{L^2} + \sup_{s,t} \|\Delta_{s,t}\|_{C^1},
\end{align*}
(\ref{eq:dxl2}) follows easily from (\ref{eq:xl2'}). 
\end{proof}

Now we can prove Proposition \ref{prop:L2}. 

\begin{proof}[\textbf{Proof of Proposition \ref{prop:L2}}]
Suppose that $u \in W^{1,3}(I \times S^1, T^*\R^n)$ satisfies 
\[
\partial_s u - J_{s,t}( \partial_t u - X_{H_{s,t}}(u))=0, \qquad
\sup_{s \in I} \big\lvert \mca{A}_{H_s}(u(s)) \big\rvert \le M. 
\]
By elliptic regularity, $u$ is $C^\infty$ on $\interior I \times S^1$. 
By the assumption on $J_{s,t}$, it is easy to see that
\[
|J_{s,t} \partial_s u|^2 \le 4 |\partial_s u|^2, \qquad
|\partial_s u|^2 \le 2 \omega_n(\partial_s u, J_{s,t} \partial_s u).
\]
By Lemma \ref{lem:L2}, the following inequality holds for any $s \in \interior I$: 
\[
\| u(s) \|_{L^2} ^2  + \| \partial_t u(s) \|_{L^2}^2 
\le c_1 \biggl(1 + \int_{S^1} 4|\partial_s u(s,t)|^2 + \partial_s H_{s,t}(u(s,t)) \, dt\biggr).
\]
The RHS is bounded by 
\begin{align*}
\int_{S^1}4 |\partial_s u(s,t)|^2 + \partial_s H_{s,t}(u(s,t)) \, dt
&\le
\int_{S^1} 8 \omega_n(\partial_s u , J_{s,t} \partial_s u) + \partial_s H_{s,t}(u(s,t)) \, dt \\
&\le -8 \partial_s \bigl(\mca{A}_{H_s}(u(s))\bigr).
\end{align*}
By similar arguments, it is easy to show that 
\[
\int_{S^1} |\partial_s u(s,t)|^2 \, dt \le -2 \partial_s \bigl(\mca{A}_{H_s}(u(s))\bigr).
\]
Therefore
\begin{align*}
&\int_I \| u(s) \|_{L^2}^2 + \| \partial_t u(s) \|_{L^2}^2 \, ds 
\le c_1
\int_I  1 - 8\partial_s(\mca{A}_{H_s}(u(s))) \, ds
\le 
c_1(3+ 16M),  \\
&\int_I \| \partial_s u(s) \|_{L^2}^2 \, ds
\le \int_I -2 \partial_s (\mca{A}_{H_s}(u(s)))  \, ds 
\le 4M.
\end{align*}
Thus we get 
\[
\int_{I \times S^1} |u(s,t)|^2 + |\partial_t u(s,t)|^2 + |\partial_s u(s,t)|^2 \, ds dt \le 3c_1 + (16c_1+4)M.
\]
This concludes the proof of Proposition \ref{prop:L2}.
\end{proof}

\subsection{Proof of Lemma \ref{lem:C0_1}, \ref{lem:C0_2}}

First notice that Lemma \ref{lem:C0_1} is a special case of Lemma \ref{lem:C0_2}. 
Hence it is enough to prove Lemma \ref{lem:C0_2}.
First we need the following lemma: 

\begin{lem}\label{lem:AH}
Suppose that $H \in C^\infty(\R \times S^1 \times T^*\R^n)$ is a homotopy from $H^-$ to $H^+$. 
Then, there exists $M>0$ which depends only on $H$ such that 
$\big\lvert \mca{A}_{H_s}(u(s)) \big\rvert \le M$ for any $s \in \R$ and 
$u \in \mca{M}_{H,J}(x_-,x_+)$, where $x_- \in \mca{P}(H^-)$, $x_+ \in \mca{P}(H^+)$. 
\end{lem}
\begin{proof} 
Since $\mca{P}(H^-)$ and $\mca{P}(H^+)$ are finite sets, 
there exists $M>0$ such that 
\[
\mca{A}_{H^-}(x) , 
\mca{A}_{H^+}(y) \in [-M, M] \qquad
\bigl(\forall x \in \mca{P}(H^-), \, \forall y \in \mca{P}(H^+) \bigr).
\]
Since $\mca{A}_{H_s}(u(s))$ is decreasing on $s$, $\mca{A}_{H_s}(u(s)) \in [-M, M]$ for any $u \in \mca{M}_{H,J}(x_-,x_+)$. 
\end{proof} 

Now we prove Lemma \ref{lem:C0_2}. In the course of the proof, constants which we do not need to be specified are denoted as "$\const$". 

\begin{proof}[\textbf{Proof of Lemma \ref{lem:C0_2}}]

To estimate $\| u \|_{C^0}$, it is enough to bound  
$\| u|_{[j,j+1] \times S^1} \|_{C^0}$ 
for each integer $j$. 
Take a cut-off function $\chi$ so that 
\[
\supp \chi \subset (-1, 2), \quad
\chi|_{[0,1]} \equiv 1, \quad 
0 \le \chi \le 1, \quad
-2 \le \chi' \le 2.
\]
Setting $v_j(s,t):= \chi(s-j) u(s,t)$, it is enough to bound $\|v_j \|_{C^0}$
(in the following, we omit the subscript $j$). 
First notice that 
\[
\| v \|_{C^0} \le 
\const \| v \|_{W^{1,3}} \le
\const \| \nabla v \|_{L^3}, 
\]
where the first inequality is a Sobolev estimate, and the second one is Poincar\'{e} inequality. 
By the Calderon-Zygmund inequality, there exists $c>0$ such that 
\[
\| \nabla v \|_{L^3} \le c \bigl( \| (\partial_s - J_\std \partial_t) v \|_{L^3} + \| v\|_{L^3} \bigr). 
\]
We claim that $\varepsilon:=1/2c$ satisfies the requirement in Lemma \ref{lem:C0_2}. 
Suppose that  $\sup_{s,t} \|J_{\std} -J_{s,t} \|_{C^0} \le 1/2c$. Then 
\begin{align*}
 c \| (\partial_s - J_\std \partial_t) v \|_{L^3} &\le c \bigl( \|J_\std -J_{s,t} \|_{C^0}\|\partial_t v\|_{L^3} + \|(\partial_s-J_{s,t} \partial_t)v \|_{L^3}\bigr)  \\
&\le \| \nabla v \|_{L^3}/2 + c \| (\partial_s - J_{s,t} \partial_t)v \|_{L^3}. 
\end{align*}
Hence we obtain 
\[
\| \nabla v \|_{L^3} \le 2c \bigl( \|v\|_{L^3} + \|(\partial_s - J_{s,t} \partial_t) v \|_{L^3} \bigr).
\]
Since $v(s,t)=\chi(s-j) u(s,t)$, it is clear that 
$\| v\|_{L^3} \le \| u \|_{L^3([j-1,j+2] \times S^1)}$. 
On the other hand, since 
\[
(\partial_s - J_{s,t} \partial_t)v(s,t)  = \chi'(s-j) u(s,t) + \chi(s-j) J_{s,t}(u) X_{H_{s,t}}(u), 
\]
and $H$ satisfies (HH3), it is easy to see 
\[
\| (\partial_s - J_{s,t} \partial_t)v \|_{L^3} \le \const ( 1 + \|u\|_{L^3([j-1,j+2] \times S^1)}).           
\]
Then we conclude that 
\[
\| \nabla v\|_{L^3} \le \const ( 1 + \|u\|_{L^3([j-1,j+2] \times S^1)}) \le \const ( 1 + \|u\|_{W^{1,2}([j-1,j+2] \times S^1)}).
\]
Then, Lemma \ref{lem:AH} and Proposition \ref{prop:L2} shows that the RHS is bounded. 
\end{proof}

\subsection{Proof of Lemma \ref{lem:C0_3}, \ref{lem:C0_4}, \ref{lem:C0_5}.}

These lemmas are consequences of the following proposition: 

\begin{prop}\label{prop:C0}
There exists a constant $\varepsilon>0$ which satisfies the following property:
\begin{quote}
Suppose we are given the following data: 
\begin{itemize}
\item $H \in C^\infty(\R \times S^1 \times T^*\R^n)$ which satisfies (HH2), (HH3). 
\item $J=(J_{s,t})_{(s,t) \in \R \times S^1}$ which satisfies (JJ2) and 
$\sup_{(s,t)} \| J_{s,t} - J_{\std} \|_{C^0} < \varepsilon$.
\item Constants $M_0, M_1>0$.  
\end{itemize}
Then, there exists a constant $c(H, M_0, M_1)>0$ such that, 
for any $\sigma \in \R$ and $u \in W^{1,3}([\sigma,\infty) \times S^1, T^*\R^n)$ satisfying 
\begin{align*}
&\partial_s u - J_{s,t}(\partial_t u - X_{H_{s,t}}(u))=0, \quad \sup_{s \ge \sigma} \big\lvert \mca{A}_{H_s}(u(s)) \big\rvert \le M_0, \\
&\| \pi(u(\sigma))\|_{W^{2/3,3}(S^1,\R^n)} \le M_1, 
\end{align*}
there holds $\| u \|_{C^0} \le c(H, M_0, M_1)$. 
\end{quote}
\end{prop}

In this subsection, we deduce Lemmas \ref{lem:C0_3}, \ref{lem:C0_4}, \ref{lem:C0_5} from Proposition \ref{prop:C0}.
First notice that Lemma \ref{lem:C0_3} is a special case of Lemma \ref{lem:C0_5}. 
Hence it is enough to prove Lemma \ref{lem:C0_4} and Lemma \ref{lem:C0_5}. 

\begin{proof}[\textbf{Proof of Lemma \ref{lem:C0_4}.}]
Since $\mca{P}(L^m)$ and $\mca{P}(H^{m+1})$ are finite sets, there exists $M>0$ such that
\[
\mca{S}_{L^m}(\gamma), \, \mca{A}_{H^{m+1}}(x) \in [-M, M]
\]
for any $\gamma \in \mca{P}(L^m)$, $x \in \mca{P}(H^{m+1})$. 
For any $(\alpha, u, v) \in \mca{N}^0(\gamma, x)$, there holds 
\begin{align*}
\mca{S}_{L^m}(\gamma) \ge \mca{S}_{L^m}(u(0)) &\ge \mca{S}_{L^{m+1}}(u(0)) \\
&\ge \mca{S}_{L^{m+1}}(u(\alpha)) \ge \mca{A}_{H^{m+1}}(v(0)) \ge \mca{A}_{H^{m+1}}(x).
\end{align*}
In particular, $\mca{S}_{L^{m+1}}(u(\alpha))$ is bounded from below. 
Now we use the following lemma: 

\begin{lem}\label{lem:bddbelow}
For any $\gamma \in \mca{P}(L^m)$ and $d \in \R$, 
$\varphi^{X^{m+1}}_{[0,\infty)}(W^u(\gamma :X^m)) \cap \{ \mca{S}_{L^{m+1}} \ge d \}$ is precompact in $\Lambda(\R^n)$. 
\end{lem}
\begin{proof}
This lemma is an immediate consequence of Proposition 2.2, Corollary 2.3 in \cite{AbM}.
Let $(\gamma_k, t_k)_{k \ge 1}$ be a sequence, 
where $\gamma_k \in W^u(\gamma :X^m)$ and $t_k \ge 0$, such that, with 
$\gamma'_k:= \varphi^{X^{m+1}}_{t_k}(\gamma_k)$, $\mca{S}_{L^{m+1}}(\gamma'_k) \ge d$. 
Since $\mca{S}_{L^m}(\gamma_k) \ge \mca{S}_{L^{m+1}}(\gamma'_k) \ge d$, 
Corollary 2.3 in \cite{AbM} shows that $(\gamma_k)_k$ has a convergent subsequence. 
Then, Proposition 2.2 (2) in \cite{AbM} implies the conclusion. 
\end{proof}

Since $\mca{S}_{L^{m+1}}(u(\alpha))$ is bounded from below for any $(\alpha,u,v) \in \mca{N}^0(\gamma,x)$, 
Lemma \ref{lem:bddbelow} shows that 
$\| u(\alpha) \|_{W^{1,2}}$ is bounded for any $(\alpha,u,v)$. Therefore, 
\[
\| \pi(v(0)) \|_{W^{2/3,3}} \le \const \| \pi(v(0)) \|_{W^{1,2}} = \const \| u(\alpha) \|_{W^{1,2}}
\]
is bounded from above (the first inequality is a Sobolev estimate). 
On the other hand $\sup_{s \ge 0} \big\lvert \mca{A}_{H^{m+1}}(v(s)) \big\rvert \le M$.  
Hence Proposition \ref{prop:C0} shows that $\| v \|_{C^0}$ is bounded. 
\end{proof}

\begin{proof}[\textbf{Proof of Lemma \ref{lem:C0_5}.}]
Suppose that $(\beta,w) \in \mca{N}^1(\gamma,x)$. Then, there holds
\[
\mca{S}_{L^m}(\gamma) \ge \mca{S}_{L^m}(\pi(w(\beta))) \ge \mca{A}_{H^m}(w(\beta)) \ge \mca{A}_{H_\beta}(w(\beta)) \ge \mca{A}_{H^{m+1}}(x).
\]
Then, $\sup_{s \ge \beta} \big\lvert \mca{A}_{H_s}(w(s)) \big\rvert$ is bounded. 
Moreover, since $\mca{S}_{L^m}(\pi(w(\beta)))$ is bounded from below, 
$\| \pi(w(\beta)) \|_{W^{1,2}}$ is bounded. 
Hence $\| w\|_{C^0}$ is bounded. 
\end{proof}

\subsection{Proof of Proposition \ref{prop:C0}}
Finally we prove Proposition \ref{prop:C0}. 

It is enough to bound  
$\sup_{(s-\sigma,t) \in [j,j+1] \times S^1} |u(s,t)|$ for each integer $j \ge 0$. 
The proof for $j \ge 1$ is as the proof of Lemma \ref{lem:C0_2}.  
Hence we only consider the case $j=0$. 
We denote the $q$-component and $p$-component of $u$ by $u_q$, $u_p$, i.e.
$u(s,t)=\bigl(u_q(s,t), u_p(s,t)\bigr)$. 

By the theory of Sobolev traces, there exists 
$\tilde{u}_q(s,t) \in W^{1,3}([\sigma,\infty) \times S^1:\R^n)$ such that 
$\tilde{u}_q(\sigma,t) = u_q(\sigma,t)$ for any $t \in S^1$, and there holds 
\[
\| \tilde{u}_q \|_{W^{1,3}([\sigma,\infty) \times S^1:\R^n)} \le 
\const \| u_q(\sigma) \|_{W^{2/3,3}(S^1:\R^n)}. 
\]
Take a cut-off function $\chi \in C^\infty([0,\infty))$ such that
\[
\supp \chi \subset [0,2), \qquad
\chi|_{[0,1]} \equiv 1 , \qquad
0 \le \chi \le 1, \qquad 
-2 \le \chi' \le 0. 
\]
We set $w(s,t):=\chi(s-\sigma)(u_q(s,t) - \tilde{u}_q(s,t), u_p(s,t))$. 
Since 
\[
\| u \|_{C^0([\sigma,\sigma+1] \times S^1)} \le \|w \|_{C^0} + \| \tilde{u}_q \|_{C^0([\sigma,\sigma+1] \times S^1)}
\le  \| w \|_{C^0} + \const \| \tilde{u}_q \|_{W^{1,3}}, 
\]
it is enough to bound $\| w \|_{C^0}$. 
It is easy to see that 
\[
\| w \|_{C^0} \le 
\const \| w \|_{W^{1,3}} \le
\const \| \nabla w \|_{L^3}
\]
by the Sobolev estimate and the Poincar\'{e} inequality. 
Since $w_q(\sigma,t)=(0,\ldots,0)$, 
we can use Calderon-Zygmund inequality to obtain 
\[
\| \nabla w \|_{L^3} \le c \bigl( \| (\partial_s - J_\std \partial_t) w \|_{L^3} + \| w \|_{L^3} \bigr). 
\]
We claim that $\varepsilon:=1/2c$ satisfies the requirement in Proposition \ref{prop:C0}. 

If $\sup_{s,t} \| J_{s,t} - J_{\std} \|_{C^0} \le 1/2c$, there holds 
\[
\| \nabla w \|_{L^3} \le 2c \bigl( \|w\|_{L^3} + \|(\partial_s - J_{s,t} \partial_t) w \|_{L^3} \bigr).
\]
We divide $(\partial_s - J_{s,t} \partial_t)w$ into two parts: 
\[
(\partial_s - J_{s,t} \partial_t)w
= \chi'(s-\sigma) (u_q-\tilde{u}_q, u_p) + \chi(s-\sigma) (\partial_s - J_{s,t} \partial_t) (u_q - \tilde{u}_q, u_p). 
\]
We bound the first and second term on the RHS: 
\begin{align*}
\|\text{first term} \|_{L^3} &\le 
\const ( \|u\|_{L^3([\sigma,\sigma+2] \times S^1)} + \| \tilde{u}_q \|_{L^3}), \\
\|\text{second term} \|_{L^3} &\le
\const( \| J_{s,t}(u) X_{H_{s,t}}(u) \|_{L^3([\sigma,\sigma+2] \times S^1)} + \| \tilde{u}_q \|_{W^{1,3}} ) \\
&\le \const ( 1+ \| u \|_{L^3([\sigma,\sigma+2] \times S^1)} + \| \tilde{u}_q \|_{W^{1,3}} ). 
\end{align*}
Hence $\| \nabla w \|_{L^3}$ is bounded by
\begin{align*}
\| \nabla w \|_{L^3}  &\le \const \bigl( 1 + \| u \|_{L^3([\sigma,\sigma+2]\times S^1)} + \| \tilde{u}_q \|_{W^{1,3}} \bigr) \\
                      &\le \const \bigl( 1 + \| u \|_{W^{1,2}([\sigma,\sigma+2] \times S^1)} + \| u_q (\sigma) \|_{W^{2/3,3}(S^1)} \bigr). 
\end{align*}
Since $\sup_{s \ge \sigma} \big\lvert \mca{A}_{H_s}(u(s)) \big\rvert$ is bounded by assumption, 
Proposition \ref{prop:L2} shows that $ \| u \|_{W^{1,2}([\sigma,\sigma+2] \times S^1)}$ is bounded. 
On the other hand, $\| u_q (\sigma) \|_{W^{2/3,3}(S^1)}$ is bounded by assumption. 
Hence the RHS is bounded.  
\qed

\section{Floer-Hofer capacity and periodic billiard trajectory} 
The goal of this section is to prove Proposition \ref{prop:billiard} and Corollary \ref{cor:FH}. 

\subsection{Symplectic homology of RCT-domains}

In this subsection, we collect some results on symplectic homology of RCT (restricted contact type) domains, 
which are essentially established in \cite{Her}. 

\begin{defn}
Let $U$ be a bounded domain in $T^*\R^n$ with a smooth boundary. 
$U$ is called RCT (restricted contact type), when there exists a vector field $Z$ on $T^*\R^n$ such that 
$L_Z \omega_n = \omega_n$ and $Z$ points strictly outwards on $\partial U$. 
\end{defn} 

Let $U$ be an RCT-domain in $T^*\R^n$. 
Then, $\mca{R}_{\partial U}:= \ker (\omega|_{\partial U})$ is a $1$-dimensional foliation on $\partial U$, which is called \textit{characteristic foliation}. 
$\mca{R}_{\partial U}$ has a canonical orientation: 
for any $p \in \partial U$, $\xi \in \mca{R}_{\partial U}(p)$ is positive if and only if 
$\omega_n(Z(p),\xi)>0$.  
$\mca{P}_{\partial U}$ denotes the set of $m$-fold coverings of closed leaves of $\mca{R}_{\partial U}$, where $m \ge 1$. 

For each $\gamma \in \mca{P}_{\partial U}$, 
$\mca{A}(\gamma):= \int_\gamma i_Z \omega_n$ is called the \textit{action} of $\gamma$. 
By our definition of orientation of $\mca{R}_{\partial U}$, $\mca{A}(\gamma)>0$ for any $\gamma \in \mca{P}_{\partial U}$. 

One can also define the Conley-Zehnder index $\ind_{\CZ}(\gamma)$ for any $\gamma \in \mca{P}_{\partial U}$, 
even when $\gamma$ is degenerate. For details, see Section 3.2 in \cite{Ir}. 
For each integer $k$, we set 
\[
\Sigma_k(\partial U):= \{ \mca{A}(\gamma) \mid \gamma \in \mca{P}_{\partial U}, \ind_{\CZ}(\gamma) \le k \} , \qquad
\Sigma(\partial U):= \bigcup_{k \in \Z} \Sigma_k(\partial U).
\]
$\Sigma(\partial U) \subset \R$ is called the \textit{action spectrum}. 

\begin{lem}\label{lem:RCT}
For any RCT-domain $U$ in $T^*\R^n$, the following statements hold: 
\begin{enumerate}
\item For any $0<a<\min \Sigma_{n+1}(\partial U)$, $\SH_n^{[-1,a)}(U) \cong \Z_2$. 
\item Let $V$ be another RCT-domain in $T^*\R^n$ such that $V \subset U$. Then, for any $a$ satisfying 
\[
0 < a < \min \Sigma_{n+1}(\partial U), \min \Sigma_{n+1}(\partial V), 
\]
the natural homomorphism $\SH_n^{[-1,a)}(U) \to \SH_n^{[-1,a)}(V)$ is an isomorphism.
\item For any $0< \varepsilon < \min \Sigma_{n+1}(\partial U)$, 
\[ 
c_{\FH}(U) = \inf \{ a \mid \text{$\SH_n^{[-1,\varepsilon)}(U) \to \SH_n^{[-1,a)}(U)$ vanishes} \}.
\]
\item $c_{\FH}(U) \in \Sigma_{n+1}(\partial U)$. 
\end{enumerate}
\end{lem}
\begin{proof}
In Proposition 4.7 in \cite{Her}, the following statement is proved: 
\begin{quote}
Let $U$ be an RCT-domain, and $0<a < \min \Sigma(\partial U)$. 
Then, $\SH_*^{[-1,a)}(U) \cong H_{n+*}(U,\partial U)$. 
\end{quote}
(1) in our Lemma \ref{lem:RCT} can be proved in the same way as this statement in \cite{Her}, 
although our assumption $a< \min \Sigma_{n+1}(\partial U)$ is weaker. 
(2) also follows directly from the proof of Proposition 4.7 in \cite{Her}. For details, see \cite{Her} pp.360 -- 361. 
(3) is Proposition 5.7 in \cite{Her}. 
(4) is proved in exactly the same way as Theorem 8 in \cite{Ir}. 
\end{proof}

\subsection{Periodic billiard trajectory}
The goal of this subsection is to prove Proposition \ref{prop:billiard}. 
Throughout this subsection, $V$ denotes a bounded domain in $\R^n$ with smooth boundary. 
First we clarify the definition of periodic billiard trajectory. 

\begin{defn}\label{def:billiard}
A continuous map $\gamma: \R/T\Z \to \bar{V}$ is called a \textit{periodic billiard trajecotory} if there exists a 
finite set $\mca{B} \subset \R/T\Z$ such that the following holds: 
\begin{itemize}
\item On $(\R/T \Z) \setminus \mca{B}$, there holds $\ddot{\gamma} \equiv 0$ and $|\dot{\gamma}| \equiv 1$. 
\item For any $t \in \mca{B}$, $\dot{\gamma}_{\pm}(t):= \lim_{h \to \pm 0} \dot{\gamma}(t+h)$ satisfies the law of reflection: 
\[
\dot{\gamma}_+(t) + \dot{\gamma}_-(t) \in T_{\gamma(t)} \partial V, \qquad
\dot{\gamma}_+(t) - \dot{\gamma}_-(t) \in (T_{\gamma(t)} \partial V)^{\perp} \setminus \{0\}. 
\]
\end{itemize}
Elements of $\mca{B}$ are called as \textit{bounce times}, and $T$ is called the \textit{length} of $\gamma$. 
\end{defn}

First we construct a sequence of RCT-domains which approximates $D^*V$. 
Fix a positive smooth function $h: V \to \R_{>0}$ and a compactly supported vector field $Z$ on $\R^n$ so that: 
\begin{itemize} 
\item $h(q) = \dist(q,\partial V)^{-2}$ when $q$ is sufficiently close to $\partial V$. 
\item $Z$ points strictly outwards on $\partial V$.
\item $dh(Z) \ge 0$ everywhere on $V$. 
\item Setting $Z = \sum_j Z_j \partial_{q_j}$, $\sup_{q \in V} | \partial_{q_i} Z_j(q)| \le 1/2n$. 
\end{itemize}
For any $\varepsilon>0$, we set $H_\varepsilon(q,p):=|p|^2 /2 + \varepsilon h(q)$, 
and $U_\varepsilon:= \{H_\varepsilon  < 1/2 \} \subset D^*V$. 

We show that $U_\varepsilon$ is an RCT-domain. 
We define $H_Z \in C^\infty(T^*\R^n)$ by $H_Z(q,p):=p \cdot Z(q)$. 
We define a vector field $\bar{Z}$ on $T^*\R^n$ by 
\[
\bar{Z}:= \sum_i p_i \partial_{p_i} + X_{H_Z}. 
\]
It is easy to check that $L_{\bar{Z}}\omega_n = \omega_n$. 
$(\varphi_t^{\bar{Z}})_t$ denotes the flow generated by $\bar{Z}$, i.e. 
$\varphi_0^{\bar{Z}} = \id_{T^*\R^n}$, and 
$\partial_t \varphi_t^{\bar{Z}} = \bar{Z}(\varphi_t^{\bar{Z}})$. 

\begin{lem}\label{lem:Z}
When $\varepsilon>0$ is sufficiently small, $dH_\varepsilon(\bar{Z})>0$ on $\{H_\varepsilon =1/2\}$. 
In particular, $U^-_\varepsilon = \{H_\varepsilon < 1/2 \}$ is an RCT-domain. 
There exists $T_\varepsilon>0$ such that 
$\varphi_{T_\varepsilon}^{\bar{Z}}(U^-_\varepsilon) \supset D^*V$. 
Moreover, we can take $T_\varepsilon$ so that $\lim_{\varepsilon \to 0} T_\varepsilon =0$. 
\end{lem}
\begin{proof}
By simple computations, 
\begin{align*}
\sum_i p_i dp_i(\bar{Z})(q,p) &= |p|^2 - \sum_{i,j} p_i p_j \partial_{q_i} Z_j(q)  \\
                              &\ge |p|^2 - \sum_{i,j} \frac{p_i^2+p_j^2}{2} \cdot \sup_{q \in V} |\partial_{q_i} Z_j(q)| \ge |p|^2/2, \\
dH_\varepsilon( \bar{Z}) (q,p) &= \sum_i p_i dp_i(\bar{Z}) (q,p) + \varepsilon dh(Z(q)) \ge |p|^2/2.
\end{align*}
Hence there holds the following claims: 
\begin{itemize}
\item $dH_\varepsilon(\bar{Z}) >0$ everywhere on $D^*V \setminus \{p=0\}$. 
\item $\sum_i p_i dp_i(\bar{Z})>0$ on $\{|p|=1\}$. 
\item $\bar{Z}$ points outwards on $\{(q,p) \mid q \in \partial V \}$. 
\end{itemize}
Since $Z$ points outwards on $\partial V$, for sufficiently small $\varepsilon>0$, $dh(Z)>0$ on 
$\{h = 1/2\varepsilon \}$.  
Hence the first property implies that $dH_\varepsilon(\bar{Z})>0$ on $\{H_\varepsilon =1/2\}$. 
By the second and third properties, $\varphi_{-T}^{\bar{Z}}(D^*V) \subset D^*V$ for any $T>0$. 
Hence, for sufficiently small $\varepsilon>0$, there holds
$\varphi_{-T}^{\bar{Z}}(D^*V) \subset U^-_\varepsilon$. This means that 
$D^*V \subset \varphi_T^{\bar{Z}}(U^-_\varepsilon)$. 
\end{proof}

By Lemma \ref{lem:Z}, there exist sequences $\varepsilon_1 > \varepsilon_2 > \cdots$, 
$T_1 > T_2 > \cdots$ such that: 
\begin{itemize}
\item $U^-_k:= \{H_{\varepsilon_k} < 1/2\}$ is an RCT-domain with respect to $\bar{Z}$. 
\item Setting $U^+_k := \varphi_{T_k}^{\bar{Z}}(U^-_k)$, there holds  $U^+_1 \supset U^+_2 \supset \cdots$ and $\bigcap_k U^+_k =\overline{D^*V}$. 
\item $\lim_{k \to \infty} \varepsilon_k= \lim_{k \to \infty} T_k = 0$. 
\end{itemize}
Since $U^+_k = \varphi_{T_k}^{\bar{Z}}(U^-_k)$ and $L_{\bar{Z}}\omega_n = \omega_n$, 
$\lim_{k \to \infty} c_\FH(U^+_k)/c_\FH(U^-_k) = \lim_{k \to \infty} e^{T_k} = 1$. 
On the other hand, $c_\FH(U^-_k) \le c_\FH(D^*V) \le c_\FH(U^+_k)$. Therefore there holds 
\[
\lim_{k \to \infty} c_\FH(U^-_k) = \lim_{k \to \infty} c_\FH(U^+_k) = c_\FH(D^*V). 
\]

\begin{lem}\label{lem:billiard}
Suppose that there exists a sequence $(\gamma_k)_k$ such that $\gamma_k \in \mca{P}_{\partial U^-_k}$, which satisfies 
$\sup_k \ind_{\CZ}(\gamma_k) \le m$ and $\lim_{k \to \infty} \mca{A}(\gamma_k) = a$, where 
$m$ is an integer and $a \ge 0$. 
Then, there exists a periodic billiard trajectory with at most $m$ bounce times and length equal to $a$. 
In particular, $a>0$. 
\end{lem}
\begin{proof}
By our assumption, 
there exists $\Gamma_k: \R/ \tau_k \Z \to \{H_{\varepsilon_k} = 1/2\}$ such that 
\[
\dot{\Gamma}_k = X_{H_{\varepsilon_k}}(\Gamma_k),\qquad  \int_{\Gamma_k} \sum_i p_i dq_i = \mca{A}(\gamma_k), 
\qquad \ind_{\CZ}(\Gamma_k) \le m. 
\]
For the last estimate, see Lemma 8 in \cite{Ir}. 
Let $q_k: \R/ \tau_k \Z \to \R^n$ be the $q$-component of $\Gamma_k$. Then, by simple computations 
\[
\ddot{q}_k + \varepsilon_k \nabla h(q_k) \equiv 0, \qquad
\frac{|\dot{q}_k|^2}{2}+ \varepsilon_k h(q_k) \equiv 1/2, \qquad 
\int_0^{\tau_k} |\dot{q}_k|^2 dt = \mca{A}(\gamma_k).
\]
Moreover, the following identity is well-known (see Theorem 7.3.1 in \cite{Long}):
\[
\ind_{\Morse} (q_k) = \ind_\CZ(\Gamma_k) \le m. 
\]

To show that $(q_k)_k$ converges to a periodic billiard trajectory on $V$, first we show that 
$\liminf_k \tau_k >0$. 
If this is not the case, by taking a subsequence, we may assume that $\lim_{k \to \infty} \tau_k =0$. 
Then, according to Proposition 2.3 in \cite{AlM}, there exists $q_\infty \in \bar{V}$ such that $(q_k)_k$ converges to the constant loop 
at $q_\infty$ in $C^0$-norm. 
However, this leads to a contradiction by the following arguments:
\begin{itemize}
\item 
Suppose $q_\infty \in V$. Let $K$ be a compact neighborhood of $q$ in $V$. 
Then, for sufficiently large $k$, $\Im q_k \subset K$. 
On the other hand, $\lim_{k \to \infty} \| \varepsilon_k h|_K \|_{C^1} =0$, and $q_k$ satisfies 
$\ddot{q_k} + \varepsilon_k \nabla h(q_k) \equiv 0$, 
$|\dot{q_k}|^2/2 + \varepsilon_k h(q_k) \equiv 1/2$.
This is a contradiction. 
\item
Suppose $q_\infty \in \partial V$. Let $\nu$ be the inward normal vector of $\partial V$ at $q_\infty$. 
For any $k$, there exists $\theta_k \in S^1$ such that $\ddot{q_k}(\theta_k) \cdot \nu \le 0$. 
On the other hand, for any $x \in V$ sufficiently close to $q_\infty$, there holds 
$\nabla h(x) \cdot \nu <0$. This contradicts our assumption that $q_k$ satisfies $\ddot{q_k} + \varepsilon_k \nabla h(q_k) \equiv 0$ for any $k$. 
\end{itemize}

Thus we have proved $\liminf_{k} \tau_k >0$. 
On the other hand, there also holds $\limsup_k \tau_k < \infty$ by exactly the same arguments as on pp. 3312 in \cite{AlM}
(see also Lemma 15 in \cite{Ir}). 

Since $(\tau_k)_k$ satisfies $0< \liminf_k \tau_k \le \limsup_k \tau_k < \infty$ and $\ind_\Morse(q_k) \le m$, 
Proposition 2.1 and 2.2 in \cite{AlM} show that a certain subsequence of $(q_k)_k$ converges to a periodic billiard trajectory on $V$ 
with at most $m$ bounce times, and length $\lim_{k \to \infty} \mca{A}(\gamma_k) = a$. 
\end{proof}

Now, the proof of Proposition \ref{prop:billiard} is immediate.

\begin{proof} 
By Lemma \ref{lem:RCT} (4), there exists $\gamma_k \in \mca{P}_{\partial U^-_k}$ such that 
$\mca{A}(\gamma_k) = c_\FH(U^-_k)$ and $\ind_\CZ (\gamma_k) \le n+1$. 
Since $\lim_{k \to \infty} c_\FH(U^-_k) = c_\FH(D^*V)$, 
Lemma \ref{lem:billiard} concludes the proof. 
\end{proof} 

\subsection{Floer-Hofer capacity} 
In this subsection, we prove Corollary \ref{cor:FH}. 
First we need the following lemma: 

\begin{lem}\label{lem:epsilon}
For sufficiently small $\varepsilon>0$, the following holds:
\begin{enumerate}
\item For sufficiently large $k$, the natural homomorphisms 
\[
\SH_n^{[-1,\varepsilon)}(U^+_k) \to \SH_n^{[-1,\varepsilon)}(D^*V), \qquad
\SH_n^{[-1,\varepsilon)}(D^*V) \to \SH_n^{[-1,\varepsilon)}(U^-_k)
\]
are isomorphisms, and the above homology groups are isomorphic to $\Z_2$.  
\item The natural homomorphism 
$H_n(\bar{V},\partial V) \to H_n \bigl(\Lambda^{<\varepsilon}(\bar{V}), \Lambda^{<\varepsilon}(\bar{V}) \setminus \Lambda(V) \bigr)$ is an isomorphism. 
\end{enumerate}
\end{lem}
\begin{proof}
Lemma \ref{lem:billiard} shows that $\liminf_k \min \Sigma_{n+1} (\partial U^-_k) >0$. 
Then, for any  
$0< \varepsilon < \liminf_k \min \Sigma_{n+1} (\partial U^-_k)$, 
(1) follows from Lemma \ref{lem:RCT} (1), (2). 
In particular, Theorem \ref{thm:main} shows that 
$H_n \bigl(\Lambda^{<\varepsilon}(\bar{V}), \Lambda^{<\varepsilon}(\bar{V}) \setminus \Lambda(V) \bigr) \cong \Z_2$. 
Hence (2) holds when 
$H_n(\bar{V},\partial V) \to H_n \bigl(\Lambda^{<\varepsilon}(\bar{V}), \Lambda^{<\varepsilon}(\bar{V}) \setminus \Lambda(V) \bigr)$ 
is injective. 

Let us fix $p \in V$. Then, if $0 < \varepsilon < 2\dist(p, \partial V)$, any $\gamma \in \Lambda^{<\varepsilon}(\bar{V}) \setminus \Lambda(V)$ satisfies 
$\gamma(S^1) \subset \bar{V} \setminus \{p\}$. Hence we get a commutative diagram 
\[
\xymatrix{
H_n(\bar{V},\partial V) \ar[r] \ar[rd] & H_n(\Lambda^{<\varepsilon}(\bar{V}), \Lambda^{<\varepsilon}(\bar{V}) \setminus \Lambda(V)) \ar[d]^{(\ev)_*} \\
                 & H_n(\bar{V}, \bar{V} \setminus \{p\})
}
\]
where the vertical arrow is induced by the evaluation map $\ev: \gamma \mapsto \gamma(0)$. 
Since the diagonal map is an isomorphism, the horizontal arrow is injective. 
\end{proof}

Finally, we prove Corollary \ref{cor:FH}. 

\begin{proof}[\textbf{Proof of Corollary \ref{cor:FH}}]
Our goal is to show that $c_\FH(D^*V)$ is equal to 
\[
b_*:=\inf \{ b \mid \text{$H_n (\bar{V}, \partial V) \to H_n \bigl(\Lambda^{<b}(\bar{V}), \Lambda^{<b}(\bar{V}) \setminus \Lambda(V) \bigr)$ vanishes} \}.
\]
Take $\varepsilon>0$ so that it satisfies the conditions in Lemma \ref{lem:epsilon}. Then, 
\begin{align*}
&\text{$H_n (\bar{V}, \partial V) \to H_n \bigl(\Lambda^{<b}(\bar{V}), \Lambda^{<b}(\bar{V}) \setminus \Lambda(V) \bigr)$ vanishes} \\
\iff&\text{$H_n \bigl(\Lambda^{<\varepsilon}(\bar{V}), \Lambda^{<\varepsilon}(\bar{V}) \setminus \Lambda(V) \bigr)
\to H_n \bigl(\Lambda^{<b}(\bar{V}), \Lambda^{<b}(\bar{V}) \setminus \Lambda(V) \bigr)$ vanishes} \\
\iff&\text{$\SH_n^{[-1,\varepsilon)}(D^*V) \to \SH_n^{[-1,b)}(D^*V)$ vanishes}. 
\end{align*}
The first equivalence follows from Lemma \ref{lem:epsilon} (2), and the second equivalence follows from Theorem \ref{thm:main}. 
Since $c_{\FH}(D^*V) = \lim_{k \to \infty} c_{\FH}(U^+_k) = \lim_{k \to \infty} c_{\FH}(U^-_k)$, it is enough to show that 
$c_{\FH}(U^+_k) \ge b_*$ and $c_{\FH}(U^-_k) \le b_*$ for any $k$. 
Since $U^-_k$ and $U^+_k$ are RCT-domains, Lemma \ref{lem:RCT} (3) implies that 
\[
c_{\FH}(U^\pm_k) = \inf \{ b \mid \text{$\SH_n^{[-1,\varepsilon)}(U^\pm_k) \to \SH_n^{[-1,b)}(U^\pm_k)$ vanishes} \}.
\]
Therefore, $c_{\FH}(U^+_k) \ge b_*$ follows from the commutativity of 
\[
\xymatrix{
\SH_n^{[-1,\varepsilon)}(U^+_k) \ar[r] \ar[d]_{\cong} & \SH_n^{[-1,b)}(U^+_k) \ar[d] \\
\SH_n^{[-1,\varepsilon)}(D^*V) \ar[r] & \SH_n^{[-1,b)}(D^*V).
}
\]
On the other hand, 
$c_{\FH}(U^-_k) \le b_*$ follows from the commutativity of 
\[
\xymatrix{
\SH_n^{[-1,\varepsilon)}(D^*V) \ar[r] \ar[d]_{\cong} & \SH_n^{[-1,b)}(D^*V) \ar[d] \\
\SH_n^{[-1,\varepsilon)}(U^-_k) \ar[r] & \SH_n^{[-1,b)}(U^-_k).
}
\]
\end{proof} 

\section{Floer-Hofer capacity and inradius} 

The goal of this section is to prove Theorem \ref{thm:inrad}. 
First of all, the lower bound $2r(V) \le c_\FH(D^*V)$ is immediate from the proof of Lemma \ref{lem:epsilon} (2). 
The upper bound $c_\FH(D^*V) \le 2(n+1) r(V)$ is a consequence of the following lemma: 

\begin{lem}\label{lem:upperbd}
For any $b> 2(n+1) r(V)$, there exists a continuous map $C: \bar{V} \times [0,1] \to \Lambda^{<b}(\bar{V})$ which satisfies the following properties: 
\begin{enumerate}
\item[(a):] For any $x \in \bar{V}$, $C(x,0)= c_x:=\text{constant loop at $x$}$. 
\item[(b):] Setting $\tilde{V}:= \partial V \times [0,1] \cup V \times \{1\}$, 
$C(\tilde{V}) \subset \Lambda^{<b}(\bar{V}) \setminus \Lambda(V)$.
\end{enumerate} 
\end{lem}

Let us check that Lemma \ref{lem:upperbd} implies the upper bound. 
By Corollary \ref{cor:FH}, 
it is enough to show that  
\[
(\iota^b)_*: H_n(\bar{V}, \partial V) \to H_n(\Lambda^{<b}(\bar{V}), \Lambda^{<b}(\bar{V}) \setminus \Lambda(V))
\]
vanishes when $b> 2(n+1) r(V)$. 
Take $C: \bar{V} \times [0,1] \to \Lambda^{<b}(\bar{V})$ as in Lemma \ref{lem:upperbd}. 
Setting $I: \bar{V} \to \bar{V} \times [0,1]$ by $I(x):=(x,0)$, consider the following diagram:
\[
\xymatrix{ 
H_*(\bar{V}, \partial V) \ar[d]_{I_*} \ar[r]^-{(\iota^b)_*} &H_*(\Lambda^{<b}(\bar{V}), \Lambda^{<b}(\bar{V}) \setminus \Lambda(V)). \\
H_*(\bar{V} \times [0,1], \tilde{V}) \ar[ru]_{C_*} &
}
\]
(a) in Lemma \ref{lem:upperbd} implies that the above diagram is commutative. 
It is easy to see that $H_*(\bar{V} \times [0,1], \tilde{V})=0$. Therefore $(\iota^b)_*=0$. 
This completes the proof of $c_\FH(D^*V) \le 2(n+1) r(V)$ modulo Lemma \ref{lem:upperbd}. 

Now our task is to prove Lemma \ref{lem:upperbd}.
Since $b> 2(n+1) r(V)$, one can take $\rho$ so that 
$\rho > 2r(V)$ and $(n+1)\rho<b$. 
We fix $\rho$ in the following argument. 

\begin{lem}\label{lem:sigmap}
For any $p \in \bar{V}$, there exists $U_p$, a neighborhood of $p$ in $\bar{V}$ and 
a continuous map $\sigma_p: U_p \times [0,1] \to \Lambda^{<\rho}(\bar{V})$ so that the following holds for any $x \in U_p$: 
\begin{itemize}
\item For any $0 \le t \le 1$, $\sigma_p(x,t)$ maps $0 \in S^1$ to $x$. 
\item $\sigma_p(x,0) = c_x$. 
\item $\sigma_p(x,1) \notin \Lambda(V)$. 
\end{itemize} 
\end{lem} 
\begin{proof}
Since $\rho/2 > r(V)$, 
there exists a smooth path $\gamma:[0,1] \to \bar{V}$ such that $\gamma(0)=p$, $\gamma(1) \in \partial V$
and length of $\gamma$ is less than $\rho/2$. 
There exists a neighborhood $U_p$ of $p$ and a continuous map 
$\Gamma: U_p \to W^{1,2}([0,1],\bar{V})$ such that 
\begin{itemize}
\item $\Gamma(p)= \gamma$. 
\item For any $x \in U_p$, $\Gamma(x)(0)=x$, $\Gamma(x)(1) \in \partial V$. 
\item For any $x \in U_p$, length of $\Gamma(x)$ is less than $\rho/2$. 
\end{itemize} 
Now define $\sigma_p: U_p \times [0,1] \to \Lambda^{<\rho}(\bar{V})$ by 
\[
\sigma_p(x,t)(\tau) = \begin{cases}
                       \Gamma(x)(2t\tau) &(0 \le \tau \le 1/2) \\
                       \Gamma(x)(2t(1-\tau)) &(1/2 \le \tau \le 1)
                      \end{cases}. 
\]
Then, it is immediate to see that $\sigma_p$ satisfies the required conditions. 
\end{proof} 

\begin{lem}\label{lem:covering}
Let $(U_p)_{p \in \bar{V}}$ be an open covering of $\bar{V}$ as in Lemma \ref{lem:sigmap}. 
Then, there exists $(W_j)_{1 \le j \le m}$, which is a refinement of $(U_p)_{p \in \bar{V}}$ and such that: 
\begin{quote}
For any $x \in \bar{V}$, the number of $j$ such that $x \in W_j$ is at most $n+1$. 
\end{quote}
\end{lem}
\begin{proof}
Actually this lemma is valid for any covering of $\bar{V}$. 
By Lebesgue's number lemma, one can take $\delta>0$ so that any subset of $\bar{V}$ with diameter less than $\delta$ is contained in some $U_p$. 
We fix such a $\delta$, and take a (smooth) triangulation $\Delta$ of $\bar{V}$ so that every simplex has diameter less than $\delta/2$. 
For each vertex $v$ of $\Delta$, $\Star(v)$ denotes the union of all open faces of $\Delta$ (we include $v$ itself), 
which contain $v$ in their closures. 

Let $v_1,\ldots, v_m$ be vertices of $\Delta$, and set $W_j:=\Star(v_j)$ for $j=1,\ldots, m$. 
Since each $W_j$ has diameter less than $\delta$, $(W_j)_{1 \le j \le m}$ is a refinement of $(U_p)_{p \in \bar{V}}$. 
Moreover, if $x \in \bar{V}$ is contained in a $k$-dimensional open face of $\Delta$, 
the number of $j$ such that $x \in W_j$ is exactly $k+1$. 
Hence $(W_j)_{1 \le j \le m}$ satisfies the required condition. 
\end{proof} 

\begin{rem}
The above proof of Lemma \ref{lem:covering} is the same as the standard proof of the fact that any $n$-dimensional polyhedron has Lebesgue
covering dimension $\le n$ (see Section 2 in \cite{Fed}). 
\end{rem} 

Take $(W_j)_{1 \le j \le m}$ as in Lemma \ref{lem:covering}. 
Since it is a refinement of $(U_p)_{p \in \bar{V}}$, 
one can define a continuous map $\sigma_j: W_j \times [0,1] \to \Lambda^{<\rho}(\bar{V})$ so that the following holds for any $x \in W_j$: 
\begin{itemize}
\item For any $0 \le t \le 1$, $\sigma_j(x,t) \in \Lambda^{<\rho}(\bar{V})$ maps $0 \in S^1$ to $x$. 
\item $\sigma_j(x,0) = c_x$. 
\item $\sigma_j(x,1) \notin \Lambda(V)$. 
\end{itemize} 

For each $1 \le j \le m$, let us take $\chi_j \in C^0(\bar{V})$ so that 
$0 \le \chi_j \le 1$, $\supp \chi_j \subset W_j$, and $K_j:= \{ x \in V \mid \chi_j(x)=1 \}$ satisfies $\bigcup_{1 \le j \le m} K_j = \bar{V}$.
We define $\tilde{\sigma}_j: \bar{V} \times [0,1] \to \Lambda^{<\rho}(\bar{V})$ by
\[
\tilde{\sigma}_j(x,t) = \begin{cases} 
                         c_x &(x \notin W_j) \\
                         \sigma_j(x, \chi_j(x) t) &( x \in W_j) 
                        \end{cases}.
\]
Then, it is immediate that $\tilde{\sigma}_j$ satisfies the following properties: 
\begin{itemize}
\item For any $x \in \bar{V}$ and $0 \le t \le 1$, $\tilde{\sigma}_j(x,t)$ maps $0 \in S^1$ to $x$. 
\item For any $x \in \bar{V}$, $\tilde{\sigma}_j(x,0)=c_x$. 
\item For any $x \in K_j$, $\tilde{\sigma}_j(x,1) \notin \Lambda(V)$. 
\end{itemize}

To finish the proof of Lemma \ref{lem:upperbd}, we introduce the following notation:

\begin{defn}\label{defn:concatanation}
For any $\gamma_1,\ldots, \gamma_m \in \Lambda(\bar{V})$ such that $\gamma_1(0)= \cdots = \gamma_m(0)$, we define their \textit{concatenation} 
$\con(\gamma_1,\ldots,\gamma_m) \in \Lambda(\bar{V})$ by 
\[
\con(\gamma_1,\ldots,\gamma_m)(t):= \gamma_{j+1}\biggl( m \biggl(t-\frac{j}{m}\biggr)\biggr) \, 
\biggl( \frac{j}{m} \le t \le \frac{j+1}{m}, \, j=0,\ldots,m-1 \biggr). 
\]
\end{defn} 

\begin{proof}[\textbf{Proof of Lemma \ref{lem:upperbd}}]
We define $C: \bar{V} \times [0,1] \to \Lambda(\bar{V})$ by 
\[
C(x,t) := \con \bigl( \tilde{\sigma}_1(x,t), \ldots, \tilde{\sigma}_m(x,t)\bigr). 
\]
Since $\tilde{\sigma}_1(x,t), \ldots, \tilde{\sigma}_m(x,t)$ maps $0 \in S^1$ to $x$, the above definition makes sense. 
We claim that this map $C$ satisfies all requirements in Lemma \ref{lem:upperbd}. 

First we have to check that length of $C(x,t)$ is less than $b$. Obviously, length of $C(x,t)$ is a sum of the lengths of 
$\tilde{\sigma}_j(x,t)$ for $j=1,\ldots,m$. 
If $x \notin W_j$, $\tilde{\sigma}_j(x,t)=c_x$ by definition. Hence $\tilde{\sigma}_j(x,t)$ has length $0$. 
Moreover, the number of $j$ such that $x \in W_j$ is at most $n+1$, by Lemma \ref{lem:covering}. 
Hence length of $C(x,t)$ is less than $(n+1) \rho < b$. 
Finally we verify conditions (a), (b). 
(a) follows from: 
\[
C(x,0) = \con \bigl( \tilde{\sigma}_1(x,0), \ldots, \tilde{\sigma}_m(x,0)\bigr) = \con(c_x, \ldots, c_x) = c_x. 
\]
To verify (b), we have to check the following two claims:
\begin{enumerate} 
\item[(b-1):] For any $x \in \bar{V}$, $C(x,1) \notin \Lambda(V)$. 
\item[(b-2):] For any $x \in \partial V$ and $0 \le t \le 1$, $C(x,t) \notin \Lambda(V)$. 
\end{enumerate} 
We check (b-1). Since $(K_j)_{1 \le j \le m}$ is a covering of $\bar{V}$, there exists $j$ such that $x \in K_j$. 
Then $\tilde{\sigma}_j(x,1) \notin \Lambda(V)$, therefore $C(x,1) \notin \Lambda(V)$. 
(b-2) is clear since $C(x,t)$ maps $0 \in S^1$ to $x \notin V$. 
\end{proof}

\textbf{Acknowledgements.}
The author would like to thank an anonymous referee for many helpful suggestions on the first draft of this paper.
The author is supported by JSPS Grant-in-Aid for Young Scientists (B) (13276352).

\end{document}